\documentclass[11pt]{article}

\usepackage{titling}

\usepackage[T1]{fontenc}
\usepackage[hmargin=2.5cm,vmargin=2.5cm]{geometry}
\usepackage{amsmath,amsthm,amssymb,mathtools}
\usepackage{thm-restate,thmtools}
\usepackage{xcolor}
\usepackage{bbm}
\usepackage[numbers]{natbib}
\usepackage{graphicx}
\usepackage{yfonts}
\usepackage[normalem]{ulem}
\usepackage{todonotes}
\usepackage[pdftex,colorlinks,linkcolor=black,urlcolor=black,citecolor=black]{hyperref}
\hypersetup{
	colorlinks=true,
    linkcolor={red!50!black},
    citecolor={blue!50!black},
    urlcolor={blue!80!black},
	bookmarksopen=true,
	bookmarksnumbered,
	bookmarksopenlevel=2,
	bookmarksdepth=3
}
\usepackage[capitalise,nameinlink]{cleveref}

\usepackage{tabularx,tablefootnote}
\usepackage{microtype}
\newtheorem{theorem}{Theorem}[section]
\newtheorem{lemma}[theorem]{Lemma}
\newtheorem{corollary}[theorem]{Corollary}
\newtheorem{proposition}[theorem]{Proposition}

\newtheorem{question}[theorem]{Question}

\theoremstyle{definition}

\newtheorem{example}[theorem]{Example}
\newtheorem{remark}[theorem]{Remark}

\newcounter{tbox}

\makeatletter
\newcommand{\leqnomode}{\tagsleft@true}
\newcommand{\reqnomode}{\tagsleft@false}
\makeatother

\newcommand*{\myproofname}{Proof}
\newenvironment{claimproof}[1][\myproofname]{\begin{proof}[#1]}{\end{proof}}

\newtheoremstyle{parens}
  {}
  {}
  {\itshape}
  {\parindent}
  {}
  {}
  {.5em}
  {(\thmnumber{#2})\thmnote{ [#3]}}

\theoremstyle{parens}
\newtheorem{nitem}[equation]{}

\crefname{theorem}{Thm.}{Thms.}
\Crefname{theorem}{Theorem}{Theorems}
\crefname{lemma}{Lem.}{Lems.}
\Crefname{lemma}{Lemma}{Lemmas}
\crefname{corollary}{Cor.}{Cors.}
\Crefname{corollary}{Corollary}{Corolarries}
\crefname{proposition}{Prop.}{Props.}
\Crefname{proposition}{Proposition}{Propositions}

\newcommand{\tw}{\mathsf{tw}}
\newcommand{\tin}{\mathsf{tree}\textnormal{-}\alpha}
\newcommand{\dg}{\mathsf{dgn}}

\newcommand{\ad}{\mathsf{ad}}

\usepackage{tikz}
\usepackage{tcolorbox}
\usepackage{enumitem}
\usetikzlibrary{arrows}
\usetikzlibrary{arrows.meta}
\usetikzlibrary{decorations.pathmorphing, decorations.pathreplacing, decorations.shapes}
\usetikzlibrary{decorations.markings}
\usetikzlibrary{calc}
\usetikzlibrary{shapes.geometric}

\pgfkeys{/tikz/.cd,
    contour distance/.store in=\ContourDistance,
    contour distance=0pt, % for the other orientation use a +
    contour step/.store in=\ContourStep,
    contour step=1pt,
}

\pgfdeclaredecoration{closed contour}{initial}
{% 
\state{initial}[width=\ContourStep,next state=cont] {
    \pgfmoveto{\pgfpoint{\ContourStep}{\ContourDistance}}
    \pgfcoordinate{first}{\pgfpoint{\ContourStep}{\ContourDistance}}
    \pgfpathlineto{\pgfpoint{0.3\pgflinewidth}{\ContourDistance}}
    \pgfcoordinate{lastup}{\pgfpoint{1pt}{\ContourDistance}}
    
  }
  \state{cont}[width=\ContourStep]{
     \pgfmoveto{\pgfpointanchor{lastup}{center}}
     \pgfpathlineto{\pgfpoint{\ContourStep}{\ContourDistance}}
     \pgfcoordinate{lastup}{\pgfpoint{\ContourStep}{\ContourDistance}}
  }
  \state{final}[width=\ContourStep]
  { % perhaps unnecessary but doesn't hurt either
    \pgfmoveto{\pgfpointanchor{lastup}{center}}
    \pgfpathlineto{\pgfpointanchor{first}{center}}
  }
}

\tikzset{
  cir/.style = {circle,draw,fill,inner sep=.7pt},
  circ/.style = {circle,draw,fill,inner sep=1.3pt},
  circg/.style = {circle,draw=lightgray,fill=lightgray,inner sep=1.3pt},
  circr/.style = {circle,draw=Crimson,fill=Crimson,inner sep=1.3pt},
  invisible/.style = {circle,draw=none,inner sep=0pt,font=\tiny},
  nonedge/.style={decorate,decoration={snake,amplitude=.3mm,segment length=1mm},draw},
}

\tcbuselibrary{skins}

\newtcolorbox{mybox}[1]{minipage boxed title*=-2cm,
enhanced,attach boxed title to top center=
{yshift=-3mm,yshifttext=-1mm},colback=Lavender!30!white,
boxed title style={size=small,colback=Lavender},coltitle=black,
center title,title={#1}}

\newcommand{\email}[1]{%
  \texttt{#1}%
}

\newcommand\extrafootertext[1]{%
    \bgroup
    \renewcommand\thefootnote{\fnsymbol{footnote}}%
    \renewcommand\thempfootnote{\fnsymbol{mpfootnote}}%
    \footnotetext[0]{#1}%
    \egroup
}

\usepackage{authblk}
\usepackage{microtype}

\title{Layered tree-independence number and clique-based separators}
\date{}

\author[1]{Clément Dallard}
\author[2]{Martin Milanič}
\author[3]{Andrea Munaro}
\author[4]{Shizhou Yang}

\affil[1]{Department of Informatics, University of Fribourg, Switzerland}
\affil[2]{FAMNIT and IAM, University of Primorska, Koper, Slovenia}
\affil[3]{Department of Mathematical, Physical and Computer Sciences, University of Parma, Italy}
\affil[4]{School of Mathematics and Physics, Queen's University Belfast, UK}

\begin{document}

\maketitle
\extrafootertext{\scriptsize Emails: 
\email{clement.dallard@unifr.ch},
\email{martin.milanic@upr.si}, 
\email{andrea.munaro@unipr.it},
\email{shizhouyang1@gmail.com}.}

\begin{abstract}
\noindent 
Motivated by a question of Galby, Munaro, and Yang (SoCG 2023) asking whether every graph class of bounded layered tree-independence number admits clique-based separators of sublinear weight, we investigate relations between layered tree-independence number, weight of clique-based separators, clique cover degeneracy and independence degeneracy. In particular, we provide a number of results bounding these parameters on geometric intersection graphs. For example, we show that the layered tree-independence number is $\mathcal{O}(g)$ for $g$-map graphs, $\mathcal{O}(\frac{r}{\tanh r})$ for hyperbolic uniform disk graphs with radius $r$, and $\mathcal{O}(1)$ for spherical uniform disk graphs with radius $r$. Our structural results have algorithmic consequences. In particular, we obtain a number of subexponential or quasi-polynomial-time algorithms for weighted problems such as $\textsc{Max Weight Independent Set}$ and $\textsc{Min Weight Feedback Vertex Set}$ on several geometric intersection graphs. Finally, we conjecture that every fractionally tree-independence-number-fragile graph class has bounded independence degeneracy.


\medskip
\noindent{\bf MSC Classes (2020):} 
05C10, % Planar graphs; geometric and topological aspects of graph theory 
05C62, % Graph representations (geometric and intersection representations, etc.)
05C75, % Structural characterization of families of graphs
05C85. % Graph algorithms (graph-theoretic aspects) 
\end{abstract}

\section{Introduction}

The main objective of this paper is to shed some light on two graph-theoretic properties implying the existence of subexponential-time algorithms for a wealth of $\mathsf{NP}$-hard graph problems: boundedness of layered tree-independence number and existence of clique-based separators of sublinear weight. These properties hold in particular for several well-known geometric intersection graphs. Recall that, given a collection $\mathcal{O}$ of objects in some metric space, its \textit{intersection
graph} is the graph whose vertices are the objects in $\mathcal{O}$ and where two distinct vertices are adjacent if and only if the
corresponding objects intersect.

We now properly define the notions of layered tree-independence number and clique-based separator. 
Let $G$ be a graph. A \emph{layering} of $G$ is an ordered partition $(V_0,V_1,\ldots)$ of the vertex set of $G$ such that for each edge $\{u,v\}\in E(G)$, if $u\in V_i$ and $v\in V_j$, then $|i-j|\le 1$.\footnote{For convenience, we consider layerings as infinite sequences of pairwise disjoint subsets of $V(G)$ with union $V(G)$, with $V_i = \varnothing$, for all sufficiently large $i$.} The sets $V_i$ are called \emph{layers}. 
The \emph{layered independence number} of a tree decomposition $\mathcal{T} = (T,\{X_t\}_{t \in V(T)})$ of $G$ is the minimum integer $\ell$ such that there exists a layering $(V_0,V_1,\ldots)$ of $G$ with $\alpha(G[X_t\cap V_i])\le \ell$, for every bag $X_t$ and layer $V_i$, where $\alpha(H)$ denotes the independence number of a graph $H$. The \emph{layered tree-independence number} of $G$ is the minimum layered independence number of a tree decomposition of $G$. If, however, the ``quality measure'' of the pair given by a tree decomposition and a layering of $G$ is given by the maximum \textsl{size} of the intersection between a bag and a layer, then we end up with the stronger notion of layered treewidth \cite{DMW17}.

As shown in~\cite{GMY24}, several interesting graph classes have bounded layered tree-independence number: intersection graphs of similarly-sized $c$-fat objects in $\mathbb{R}^2$ (in particular, unit disk graphs), intersection graphs of unit-width rectangles in $\mathbb{R}^2$, odd powers of graphs of bounded layered tree-independence number,\footnote{This does not extend to even powers \cite{GMY24}.} and some subclasses of (vertex and edge) intersection graphs of paths on a grid. 
Note, however, that layered tree-independence number fails to capture geometric intersection graphs in higher dimensions, as unit ball graphs in $\mathbb{R}^3$ have unbounded layered tree-independence number \cite{GMY24}.  

The \textit{tree-independence number} of a graph $G$ is the minimum independence number over all tree decompositions of $G$, where the \textit{independence number} of $\mathcal{T} = (T,\{X_t\}_{t \in V(T)})$ is the quantity $\max_{t\in V(T)} \alpha(G[X_t])$.
Bounded layered tree-independence number of a class $\mathcal{G}$ implies $\mathcal{O}(\sqrt{n})$ tree-independence number for $n$-vertex graphs in $\mathcal{G}$ \cite{GMY24} and hence, by Ramsey's theorem, $\mathcal{O}(\sqrt{n})$ treewidth for $n$-vertex graphs in $\mathcal{G}$ with bounded clique number (see also \Cref{ramsey,treealphabound}). 
Combining this observation with the $n^{\mathcal{O}(k)}$-time algorithms for \textsc{Max Weight Distance-$d$ Packing} (for fixed even $d \in \mathbb{N}$) and \textsc{Min Weight Feedback Vertex Set} for graphs with tree-independence number at most $k$ \cite{LMMORS24} gives $2^{\mathcal{O}(\sqrt{n}\log n)}$-time algorithms for the aforementioned graph classes.

Let us now introduce clique-based separators. A pair of vertex subsets $(A,B)$ is a \textit{separation} in a graph $G$ if $A \cup B = V(G)$ and there is no edge between $A \setminus B$ and $B \setminus A$. A separation $(A,B)$ is $\beta$-\textit{balanced} if $\max\{|A \setminus B|, |B \setminus A|\} \leq \beta|V(G)|$ for some $\beta < 1$. The \textit{separator} of a separation $(A, B)$ is the set $A \cap B$. A \textit{clique-based separator} of a graph $G$ is a collection $\mathcal{S}$ of vertex-disjoint cliques whose union is a balanced 
separator of $G$. The \textit{size} of $\mathcal{S}$ is $|\mathcal{S}|$ and the \textit{weight} of $\mathcal{S}$ is the quantity $\sum_{C\in\mathcal{S}} \log(|C|+1)$. De Berg et al.~\cite{BBK20,dBKMT23} showed that several intersection graphs of geometric objects (e.g., $0$-map graphs, intersection graphs of convex fat objects or similarly-sized fat objects in $\mathbb{R}^d$, and pseudo-disk graphs) admit clique-based separators of sublinear weight (sublinear in the number of vertices) and used this to obtain subexponential-time algorithms for many \textit{unweighted} problems on such graphs. For example, if $\mathcal{G}$ is a class of geometric intersection graphs such that, for every $n$-vertex graph $G \in \mathcal{G}$ we can compute in polynomial time a clique-based separator of $G$ of weight $w_{\mathcal{G}}(n)$, then \textsc{Independent Set} and \textsc{Feedback Vertex Set} admit $2^{\mathcal{O}(w_{\mathcal{G}}(n))}$-time algorithms \cite{dBKMT23}.

Observe that, if a graph class $\mathcal{G}$ admits clique-based separators of sublinear weight, then it also admits separators of sublinear independence number and so has sublinear tree-independence number, thanks to the following observation based on \cite[Theorem~20]{Bod98} (see also \cite[Lemma~2]{ACO23}). For completeness, we provide a proof in \Cref{s:prelim}. 

\begin{restatable}{lemma}{observ}
\label{cliquesep_totreealpha}
Let $0 < \beta < 1$ and let $f\colon \mathbb{R}_{>0}\to \mathbb{R}_{\geq 0}$ be a non-decreasing function. If $\mathcal{G}$ is a hereditary\footnote{A graph class is hereditary if it is closed under vertex deletion.} graph class such that, for every $n$-vertex graph $G \in \mathcal{G}$, there exists a $\beta$-balanced clique-based separator of $G$ of size $\mathcal{O}(f(n))$, then $G$ has tree-independence number $\mathcal{O}(f(n) \cdot \log{n})$. If in addition $f(n) = \Omega(n^{\varepsilon})$ for some $\varepsilon > 0$ and there exists $c < 1$ such that $f(\lfloor\beta n\rfloor) \leq cf(n)$ for all sufficiently large $n$,
then $G$ has tree-independence number $\mathcal{O}(f(n))$.  
\end{restatable}

However, there exist classes admitting separators of sublinear independence number but with no clique-based separators of sublinear weight, as shown by the following example from \cite{GMY24}. 
Let $\mathcal{G}$ be a class of arbitrarily large triangle-free graphs on $n$ vertices with independence number $\mathcal{O}(\sqrt{n\log n})$, and let $\mathcal{G}'$ be the class of all graphs obtained as the join\footnote{The join of two vertex-disjoint graphs $G_1 = (V_1, E_1)$ and $G_2 = (V_2, E_2)$ is the graph with vertex set $V_1 \cup V_2$ and edge set $E_1 \cup E_2 \cup \{v_1v_2 : v_1 \in V_1, v_2 \in V_2\}$.} of two graphs in $\mathcal{G}$. Every graph in $\mathcal{G}'$ has separators of sublinear independence number. On the other hand, each clique-based separator of a graph in $\mathcal{G}'$ has size (and hence weight) $\Omega(n)$. Galby et al.~\cite{GMY24} then asked whether the existence of clique-based separators of sublinear weight could be implied by the stronger requirement of having bounded (by a constant) layered tree-independence number (see \Cref{reldiagram}).

\begin{question}[Galby et al.~\cite{GMY24}]\label{questionltreealpha} Is it true that every graph class of bounded layered tree-independence number admits clique-based separators of sublinear weight?
\end{question}

It should be noted that admitting clique-based separators of sublinear weight does not imply having bounded layered tree-independence number. Indeed, Euclidean disk graphs (i.e., intersection graphs of disks in $\mathbb{R}^2$) admit clique-based separators of weight $\mathcal{O}(\sqrt{n})$ \cite{BBK20} but have unbounded layered tree-independence number \cite{GMY23}. This also shows that, contrary to clique-based separators, the notion of layered tree-independence number nicely separates unit disk graphs, for which the value is bounded, from disk graphs, for which the value is unbounded.  

We provide positive evidence for \Cref{questionltreealpha} by proving two types of results. 
First, we extend the known classes with bounded layered tree-independence number by showing that they include map graphs, powers of bounded layered treewidth graphs, and two generalizations of unit disk graphs, namely hyperbolic uniform disk graphs and spherical uniform disk graphs (see below for definitions). Second, we show that graphs in these classes all have bounded (by a constant) clique cover degeneracy, where the \textit{clique cover degeneracy} of a graph $G$ is the smallest integer $k$ such
that every nonnull induced subgraph $H$ of $G$ has a vertex whose neighborhood in $H$ induces a subgraph with clique cover number at most $k$. Finally, we argue that combining these two types of results implies the existence of clique-based separators of sublinear weight for each of the corresponding classes. 

Note, however, that there exist classes with bounded layered tree-independence number and unbounded clique cover degeneracy, as the following example shows.\footnote{We shall see that the graph classes in \Cref{ccdvsid} do admit clique-based separators of sublinear weight.}

\begin{example}\label{ccdvsid} Let $\mathcal{G}$ be a class of graphs with clique number at most $k$ and arbitrarily large chromatic number, and let $\mathcal{G}'$ be the class of all graphs obtained as the join of two graphs the complements of which are in $\mathcal{G}$. Clearly, every graph in $\mathcal{G}'$ has independence number at most $k$, and hence tree-independence number at most $k$. On the other hand, for any $\ell\in\mathbb{N}$, if $G_\ell$ is a graph in $\mathcal{G}$ with chromatic number at least $\ell$, then the join of two copies of the complement of $G_\ell$ is a graph in $\mathcal{G}'$ with clique cover degeneracy at least $\ell$.
\end{example}

Replacing the clique cover number with the smaller independence number in the definition of clique cover degeneracy leads to the notion of \textit{independence degeneracy}, which has been investigated under different names in several papers \cite{AADK02,CQ23,KT14,YB12}. 
An easy observation (see \Cref{bounded layered tree-alpha bounded alpha-degeneracy}) shows that bounded layered tree-independence number implies bounded independence degeneracy.
We conjecture that bounded independence degeneracy holds, more generally, for every fractionally $\tin$-fragile graph class, whose definition we postpone to \Cref{s:prelim}.  

\begin{restatable}{conjecture}{conj}
\label{fractional tin-fragility bounded independence degeneracy}
Every fractionally $\tin$-fragile graph class has bounded independence degeneracy. 
\end{restatable}

\Cref{fractional tin-fragility bounded independence degeneracy} is algorithmically motivated by the following results. In general, PTASes (polynomial-time approximation schemes) for maximization problems are not to be expected on classes of bounded independence degeneracy, as \textsc{Max Weight Independent Set} is $\mathsf{MAX \ SNP}$-hard on subcubic graphs \cite{BF99}. However, these classes allow for constant-factor approximation algorithms for a number of problems, such as \textsc{Max Weight Independent Set}, as shown by Akcoglu et al.~\cite{AADK02}, \textsc{Max Weight Induced
$q$-Colorable Subgraph} and \textsc{Min Weight Vertex Cover}, as shown by Ye and Borodin~\cite{YB12}, and further generalizations of \textsc{Max Weight Independent Set}, as shown by Chekuri and Quanrud~\cite{CQ23}. On the other hand, if one considers fractionally $\tin$-fragile graph classes, then PTASes for a large number of maximization problems exist. Galby et al.~\cite{GMY24} showed that \textsc{$(c, h, \psi)$-Max Weight Induced Subgraph} admits a PTAS on every efficiently fractionally $\tin$-fragile graph class. 
Loosely speaking, this is the problem of finding a maximum-weight induced subgraph with clique number at most $c$ satisfying a fixed $\mathsf{CMSO}_2$ formula $\psi$ expressing an $h$-near-monotone property (the special case $h = 1$ is that of a monotone property). This meta-problem captures several well-known problems, such as \textsc{Max Weight Independent Set} and, more generally, \textsc{Max Weight Induced
$q$-Colorable Subgraph}, and \textsc{Max Weight Induced Forest}.

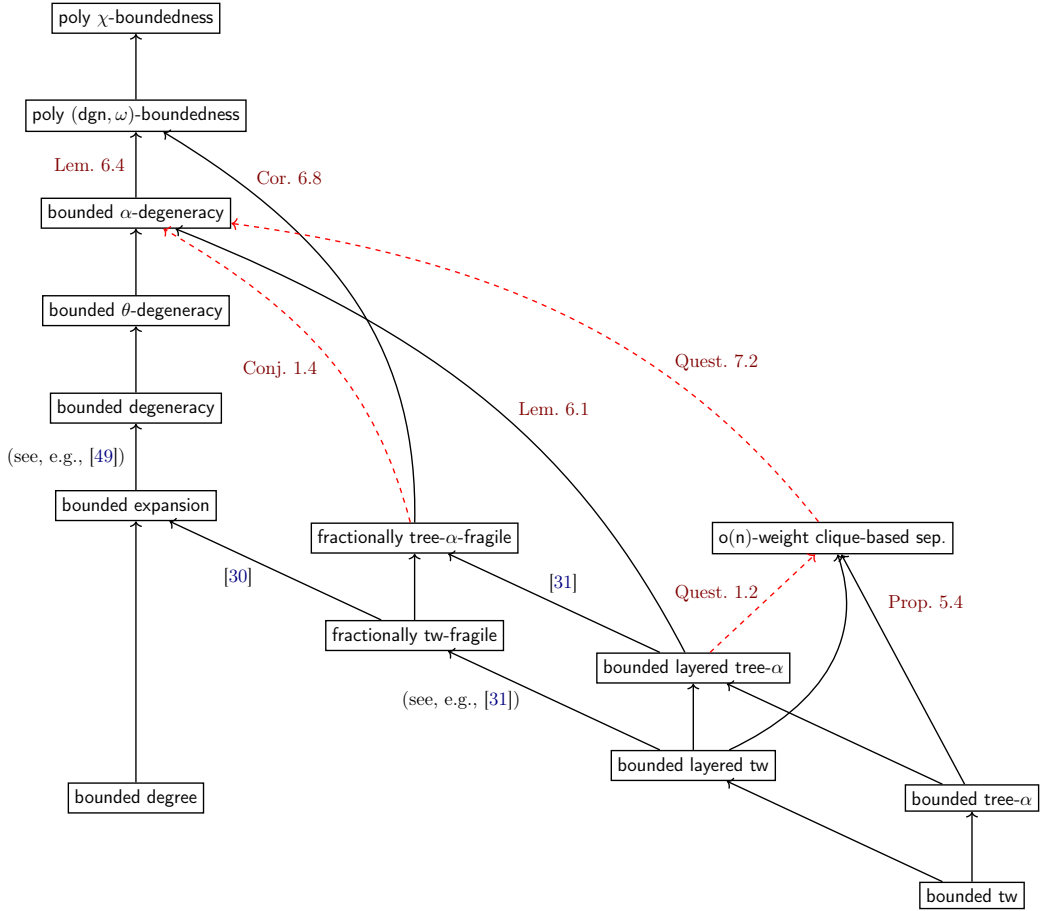
\begin{figure}
\begin{center}
\scalebox{0.65}{\begin{tikzpicture}[scale=.95,yscale=0.7, thick]
\node[rectangle,draw] (ftf) at (10,12) {\textsf{fractionally $\mathsf{tw}$-fragile}};
\node[rectangle,draw] (ftaf) at (10,15) {\textsf{fractionally $\mathsf{tree}\textnormal{-}\alpha$-fragile}};
\node[rectangle,draw] (cbsep) at (19,15) {\textsf{$\mathsf{o(n)}$-weight clique-based sep.}};
\node[rectangle,draw] (layeredtin) at (16,11) {\textsf{bounded layered $\mathsf{tree}\textnormal{-}\alpha$}};
\node[rectangle,draw] (layeredtw) at (16,8) {\textsf{bounded layered $\mathsf{tw}$}};
\node[rectangle,draw] (tin) at (22,7) {\textsf{bounded $\mathsf{tree}\textnormal{-}\alpha$}};
\node[rectangle,draw] (tw) at (22,4) {\textsf{bounded $\mathsf{tw}$}};
\node[rectangle,draw] (exp) at (4,16) {\textsf{bounded expansion}};
\node[rectangle,draw] (degn) at (4,19) {\textsf{bounded degeneracy}};
\node[rectangle,draw] (tdegn) at (4,22) {\textsf{bounded $\theta$-degeneracy}};
\node[rectangle,draw] (adegn) at (4,25) {\textsf{bounded $\alpha$-degeneracy}};
\node[rectangle,draw] (dgnomega) at (4,28) {\textsf{poly $(\dg, \omega)$-boundedness}};
\node[rectangle,draw] (chibound) at (4,31) {\textsf{poly $\chi$-boundedness}};
\node[rectangle,draw] (bdegree) at (4,7) {\textsf{bounded degree}};
\node (layeredtinfractin) at (13.2,13.6) {\cite{GMY23}};
\node (layeredtwfractw) at (11,10) {(see, e.g., \cite{GMY23})};
\node (fractwexp) at (6.2,13.8) {\cite{Dvo16}};
\node (layeredtinalphadeg) at (13,19) {\hyperref[bounded layered tree-alpha bounded alpha-degeneracy]{\textcolor{red!50!black}{Lem.~\ref*{bounded layered tree-alpha bounded alpha-degeneracy}}}};
\node (fractinpolydgnomega) at (7.3,26.1) {\hyperref[fractional tin-fragility poly chi-bounded]{\textcolor{red!50!black}{Cor.~\ref*{fractional tin-fragility poly chi-bounded}}}};
\node (alphadegpolydgnomega) at (3,26.5) {\hyperref[alpha-degeneracy dg-omega-boundedness]{\textcolor{red!50!black}{Lem.~\ref*{alpha-degeneracy dg-omega-boundedness}}}};
\node (expdeg) at (2.5,17.5) {(see, e.g., \cite{NO12})};
\node (conj) at (7.1,20.2) {\hyperref[fractional tin-fragility bounded independence degeneracy]{\textcolor{red!50!black}{Conj.~\ref*{fractional tin-fragility bounded independence degeneracy}}}};
\node (prop) at (21,13) {\hyperref[layeredtreealpha-cbs]{\textcolor{red!50!black}{Prop.~\ref*{layeredtreealpha-cbs}}}};
\node (question) at (16.5,13.3) {\hyperref[questionltreealpha]{\textcolor{red!50!black}{Quest.~\ref*{questionltreealpha}}}};
\node (question) at (16.5,20.4) {\hyperref[clique-based_vs_ideg]{\textcolor{red!50!black}{Quest.~\ref*{clique-based_vs_ideg}}}};

\draw[->](dgnomega) -- (chibound);
\draw[->](adegn) -- (dgnomega);
\draw[->](tdegn) -- (adegn);
\draw[->](degn) -- (tdegn);
\draw[->](exp) -- (degn);
\draw[->](layeredtin) -- (ftaf); 
\draw[->](layeredtw) -- (ftf);
\draw[->](tin) -- (cbsep);
\draw[->,dashed,red](layeredtin) -- (cbsep);
\draw[->](tin) -- (layeredtin); 
\draw[->](layeredtw) -- (layeredtin); 
\draw[->](ftf) -- (ftaf);
\draw[->](ftaf) to [bend right=25] (dgnomega);
\draw[->,dashed,red](cbsep) to [bend right=25] (adegn);
\draw[->,dashed,red](ftaf) to [bend right=20] (adegn);
\draw[->](layeredtw) to [bend right=35] (cbsep);
\draw[->](tw) -- (tin);
\draw[->](tw) -- (layeredtw); 
\draw[->](layeredtin) to [bend right=20] (adegn);
\draw[->](bdegree) -- (exp);
\draw[->](ftf) -- (exp);
\end{tikzpicture}}
\end{center}
\caption{Relationships between the main graph class properties related to the paper, where an arrow represents implication between class properties. The following shorthands are adopted. $\mathsf{tw}$ and $\mathsf{tree}\textnormal{-}\alpha$ are shorthands for treewidth and tree-independence-number, respectively. $\theta$-degeneracy and $\alpha$-degeneracy stand for clique cover degeneracy and independence degeneracy, respectively. We reference only the implications not directly following from the definition. A dashed red arrow indicates that determining whether the corresponding implication holds is, to the best of our knowledge, open.
As for the incomparabilities, the following hold. The class of stars shows that \textsf{bounded $\mathsf{tw}$ $\nsubseteq$ bounded degree}. The class of complete graphs shows that \textsf{bounded $\mathsf{tree}\textnormal{-}\alpha$ $\nsubseteq$ fractionally $\mathsf{tw}$-fragile}. The class of subcubic expanders shows that \textsf{bounded degree $\nsubseteq$ fractionally $\mathsf{tree}\textnormal{-}\alpha$-fragile} and \textsf{bounded degree $\nsubseteq$ $\mathsf{o(n)}$-weight clique-based sep.} (see also \cite[Corollary~38]{GMY24}). The classes in \Cref{ccdvsid} show that \textsf{bounded $\mathsf{tree}\textnormal{-}\alpha$ $\nsubseteq$ bounded $\theta$-degeneracy}.}\label{reldiagram}
\end{figure}

Before properly stating our results, we define the main graph classes addressed in the paper. 

\textit{Bounded genus map graphs.} The class of map graphs has been introduced by Chen et al.~\cite{CGP02} as a generalization of planar graphs. We are interested in a further generalization introduced by Dujmovi\'c et al.~\cite{DEW17}.
Start with a graph $G_0$ embedded in a surface of Euler genus $g$,\footnote{The Euler genus of an orientable surface with $h$ handles is $2h$. 
The Euler genus of a nonorientable surface with $c$ cross-caps is $c$. 
The Euler genus of a graph $G$ is the minimum Euler genus of a surface in which $G$ embeds (with no crossings).} with each face labeled a `nation' or a `lake'. 
Let $G$ be the graph whose vertices are the nations of $G_0$, where two vertices are adjacent in $G$ if the corresponding faces in $G_0$ share a vertex. 
Then $G$ is called a \emph{$g$-map graph}, or simply \textit{map graph} in case $g=0$. 
If, in addition, each vertex of $G_0$ is incident with at most $d$ nations, then $G$ is a \emph{$(g, d)$-map graph}. 
Observe that $g$-map graphs can be equivalently defined as a certain type of \textit{contact graphs}, i.e., intersection graphs of geometrical objects with pairwise disjoint interiors. 
For example, $0$-map graphs are contact graphs of disk homeomorphs in $\mathbb{R}^2$.

Dujmovi\'c et al.~\cite{DEW17} showed that every $(g, d)$-map graph has layered treewidth at most $(2g + 3)(2d + 1)$. This result was strengthened in \cite{DMW23} to show that $(g, d)$-map graphs (as well as several other generalizations of planar graphs) admit the following type of product structure theorem: every graph in the class is isomorphic to a subgraph of the strong product $H \boxtimes P \boxtimes K_\ell$, where $H$ has bounded treewidth, $P$ is a path, and $\ell$ is a constant depending only on the class. 

\textit{Graph powers.} For $d \in \mathbb{N}$, the \textit{$d$-th power} of a graph $G$ is the graph $G^d$ with vertex set $V(G^d) = V(G)$ where, for distinct $u, v \in V(G^d)$, $uv \in E(G^d)$ if and only if $u$ and $v$ are at
distance at most $d$ in $G$. Dujmovi\'c et al.~\cite{DEMWW22} showed that if $G$ is a graph with layered treewidth $k$ and maximum degree $\Delta$, then $G^d$ has layered treewidth less than $2dk\Delta^{\lfloor d/2\rfloor}$. (Let us remark that powers of graphs with bounded layered treewidth and bounded degree also admit a type of product structure theorem as above \cite{DMW23}.) 
This result in turn implies almost-tight bounds for the degeneracy of powers of planar graphs and chordal graphs of bounded degree \cite{AH03,Kra04}. Galby et al.~\cite{GMY24} showed that if $G$ is a graph with layered tree-independence number $k$ and $d$ is a positive integer, then $G^{2d+1}$ has layered tree-independence number at most $(4d+1)k$. Moreover, a similar result cannot hold for even powers.

\textit{Unit disk graphs and generalizations.} 
Intersection graphs of disks of common radius $r$ in the Euclidean plane, known as (Euclidean) unit disk graphs, form a well-studied and structurally rich class of intersection graphs, which we denote by $\text{UDG}$. As already mentioned, such graphs have bounded layered tree-independence number (more precisely, at most $3$), and hence $\mathcal{O}(\sqrt{n})$ tree-independence number, which is tight. In this paper, we consider two interesting generalizations of unit disk graphs, as we explain next.

It is known (see, e.g., \cite[Theorem~11.12]{Lee}) that every complete, simply-connected Riemannian surface of constant sectional curvature $C$ is isometric to either the Euclidean plane (when $C = 0$), or the radius-$R$ sphere (when $0 < C = 1/R^2$), or the radius-$R$ hyperbolic plane (when $0 > C = -1/R^2$). It is then natural to investigate analogues of unit disk graphs on the hyperbolic plane and on the unit sphere. Uniform disk graphs on the hyperbolic plane have been studied by Kisfaludi{-}Bak~\cite{KB20} and 
Bl{\"{a}}sius et al.~\cite{BvdH24}, and the study of uniform disk graphs on the unit sphere was stated as an open problem in \cite{BvdH24}. As we shall see, in contrast to the Euclidean setting, both on the hyperbolic plane and on the unit sphere, different values of the (uniform) radius give rise to different graph classes. Intuitively, the area of a Euclidean disk grows quadratically in the radius, whereas those of a hyperbolic disk and spherical disk grow exponentially and sublinearly in the radius, respectively. 

\textit{Hyperbolic uniform disk graphs.} 
Following \cite{BvdH24}, let $\text{HUDG}(r(n))$ be the class of intersection graphs where each $n$-vertex graph can be realized as the
intersection graph of disks of uniform (i.e., equal) radius $r(n)$ in the hyperbolic plane $\mathbb{H}^2$ of curvature $-1$. 
As mentioned above, the (uniform) radius matters in the hyperbolic plane. Loosely speaking, the smaller the radius the closer the corresponding graph class is to Euclidean unit disk graphs. For example, the classes of hyperbolic uniform disk graphs of radius $1/n^3$ and hyperbolic uniform disk graphs of radius $\log n$ are incomparable, and graphs distinguishing the two are large grids and large stars \cite{BvdH24}. Bl{\"{a}}sius et al.~\cite{BvdH24} referred to the regime $r \in \mathcal{O}(1/\sqrt{n})$ as \textit{almost Euclidean}, whereas the regime $r \in \Omega(\log n)$ is referred to as \textit{firmly hyperbolic}.

Let $\text{HUDG} = \bigcup_{r > 0}\text{HUDG}(r)$ and let $\text{DG}$ denote the class of Euclidean disk graphs. Bl{\"{a}}sius et al.~\cite{BvdH24} observed that $\text{UDG} \subsetneq \text{HUDG} \subsetneq \text{DG}$, where a family of hyperbolic uniform disk graphs which are not unit disk graphs is that of stars. Kisfaludi{-}Bak~\cite{KB20} showed that, for any constant $c$, the class $\text{HUDG}(c)$ admits clique-based separators of size $\mathcal{O}(\log n)$ and, hence, weight $\mathcal{O}(\log^2 n)$. Bl{\"{a}}sius et al.~\cite{BvdH24} then asked the following natural question: How does the structure of hyperbolic uniform disk graphs depend on the radius $r = r(n)$?
In view of this, they extended the previous result to arbitrary radiuses and showed that $n$-vertex hyperbolic uniform disk graphs of radius $r$ admit clique-based separators of size $\mathcal{O}((1 + 1/r)\log n)$. Applying this result to the machinery from \cite{dBKMT23}, they provided an $n^{\mathcal{O}((1 + 1/r)\log^2 n)}$-time algorithm for \textsc{Independent Set}, which is quasi-polynomial for $r \in \Omega(1)$. Moreover, if $r \in \mathcal{O}(1)$, they also provided $n^{\mathcal{O}((1/r)\log n)}$-time algorithms for a number of problems, including \textsc{Independent Set}, \textsc{Feedback Vertex Set}, and \textsc{Dominating Set}. In the case of \textsc{Independent Set}, they were finally able to improve the above to an $n^{\mathcal{O}(1 + (1/r)\log n)}$-time algorithm, which is polynomial for $r \in \Omega(\log n)$.

\textit{Spherical uniform disk graphs.} Similarly to the hyperbolic case, different radiuses give rise to different classes of spherical uniform disk graphs. 
Let $\text{SUDG} = \bigcup_{r > 0}\text{SUDG}(r)$. It can be shown that $\text{UDG} \subsetneq \text{SUDG}$, where the inclusion follows from \cite{JPL21}, and the complete bipartite graph $K_{2,3}$ is easily seen to belong to $\text{SUDG}$ but not to $\text{UDG}$.
In fact, in \Cref{spherical-superclass}, we provide an infinite family of spherical uniform disk graphs which are not unit disk graphs. Note also that there exist graphs which belong to $\text{DG}$ but not to $\text{SUDG}$ (just take $K_{1,t}$ for $t \geq 6$ \cite{MM}). However, it is not clear whether $\text{SUDG} \subseteq \text{DG}$ (see \Cref{quest:SUDG-DG}). Previous research on spherical (uniform) disk graphs has focused on the randomized setting where the disks are randomly sampled (see, e.g., \cite{Mae03,Mae04,MM17} and references therein) and, to the best of our knowledge, little is known about the classic deterministic setting (see \cite{MM}).

\subsection{Our results}

\begin{table}[t]
\centering
\renewcommand{\arraystretch}{1.15}
\begin{tabular}{|l l l l |}
 \hline
  & layered $\tin$ & $\tin$ & size of clique-based sep. \\ [0.5ex]
 \hline\hline
  $g$-map & $\mathcal{O}(g)$ [\cref{maplayered}] & $\mathcal{O}(\sqrt{gn})$ [\cref{treealphamap}] & $\mathcal{O}(\sqrt{n})$ for $g = 0$ \cite{dBKMT23}\\
  UDG & $\mathcal{O}(1)$ \cite{GMY23} & $\mathcal{O}(\sqrt{n})$ \cite{GMY23} &  $\mathcal{O}(\sqrt{n})$ \cite{BBK20} \\
  $\text{HUDG}(r)$ & $\mathcal{O}(\frac{r}{\tanh r})$ [\cref{layeredhyperbolic}] & $\mathcal{O}(\sqrt{n})$ [\cref{treealphahudg}]
  &  $\mathcal{O}(\sqrt{n})$ \cite{BvdH24,BBK20}\\
  & & $\mathcal{O}((1 + \frac{1}{r})\log^2 n)$ [\cref{cliquesep_totreealpha}] & $\mathcal{O}((1+\frac{1}{r})\log n)$ \cite{BvdH24}\\
  $\text{SUDG}(r)$ & $\mathcal{O}(1)$ [\cref{sphericallayered}] & $\mathcal{O}(\sqrt{n})$ [\cref{treealphasudg}] & $\mathcal{O}(\sqrt{n})$ [\cref{cliquebasedcor}] \\
  $\text{DG}$ & $\omega(1)$ \cite{GMY24} & $\mathcal{O}(\sqrt{n})$ [\cref{cliquesep_totreealpha}] & $\mathcal{O}(\sqrt{n})$ \cite{BBK20} \\
  contact segment & $\omega(1)$ [\cref{contactlocal}] & $\mathcal{O}(n^{2/3})$ [\cref{cliquesep_totreealpha}] & $\mathcal{O}(n^{2/3})$ \cite{BBGR24,dBKMT23} \\
\hline
\end{tabular}

\caption{A summary of structural results for geometric intersection graphs. Note that if a clique-based separator has size $\mathcal{O}(f(n))$, then it has weight $\mathcal{O}(f(n)\log n)$. The $\mathcal{O}(\sqrt{n})$ bound on clique-based separators for hyperbolic uniform disk graphs comes from a bound for the more general disk graphs. Similarly, the bound on clique-based separators for contact segment graphs comes from a bound for the more general pseudo-disk graphs. We remark that the upper bounds for disk graphs hold, more generally, for intersection graphs of convex fat objects or similarly-sized fat objects in $\mathbb{R}^d$ \cite{BBK20}.}
\label{summarytable}
\vspace*{-0.1cm}
\end{table}

\textit{Layered tree-independence number.} Our first set of results, proved in \Cref{s:layeredtreealpha}, consists in bounding the layered tree-independence number of the following graph classes (see also \Cref{summarytable}).
\begin{enumerate}[label={\selectfont\textbf{(\Alph*)}},ref=\Alph*]
\item\label{item:map} $g$-map graphs have $\mathcal{O}(g)$ layered tree-independence number;
\item\label{item:power} Powers of bounded layered treewidth graphs have $\mathcal{O}(1)$ layered tree-independence number; 
\item\label{item:hyper} Hyperbolic uniform disk graphs with radius $r$ have $\mathcal{O}(\frac{r}{\tanh r})$ layered tree-independence number;
\item\label{item:spher} Spherical uniform disk graphs with radius $r$ have $\mathcal{O}(1)$ layered tree-independence number.
\end{enumerate}

Results \eqref{item:map} and \eqref{item:power} are in fact special cases of a more general result of independent interest, which we describe next. We introduce a graph operation, called $r$-neighborhood cliquification, which, given a graph $G$ and a subset $P \subseteq V(G)$, consists in cliquifying the $r$-th neighborhood of every vertex $p \in P$, where cliquifying a vertex subset adds all missing edges between vertices in the subset. 
We show that $r$-neighborhood cliquifications of graphs with bounded layered treewidth result in graphs with bounded layered tree-independence number and that, if $P$ is in addition an independent set, $1$-neighborhood cliquifications of graphs with bounded layered tree-independence number result in graphs with bounded layered tree-independence number.

Observe that none of Results \eqref{item:map} to \eqref{item:spher} can be strengthened to bounded layered treewidth, as the corresponding classes contain arbitrarily large cliques.
Moreover, since $n$-vertex graphs with layered tree-independence number $k$ have tree-independence number $2\sqrt{kn}$ (see \Cref{treealphabound}), we immediately obtain bounds for tree-independence number, as summarized in \Cref{summarytable}.
The bounds we obtain for $g$-map graphs and spherical uniform disk graphs are tight up to constant factors, as both classes contain $n\times n$ grids,\footnote{The $n\times n$ grid is the graph with vertex set $\{1,\ldots, n\}^2$, in which two vertices $(x_1,x_2)$ and $(y_1,y_2)$ are adjacent if and only if $|x_1-y_1|+|x_2-y_2| = 1$.} which have $\Omega(\sqrt{n})$ tree-independence number \cite{GMY24}. The case of hyperbolic uniform disk graphs deserves further comments (see also \Cref{hudgconstantlayered}). On the one hand, being a subclass of Euclidean disk graphs, they have $\mathcal{O}(\sqrt{n})$ tree-independence number. On the other hand, applying \Cref{cliquesep_totreealpha} to the clique-based separators of size $\mathcal{O}((1 + 1/r)\log n)$ from \cite{BvdH24}, we get a new bound of $\mathcal{O}((1 + 1/r)\log^2 n)$ for tree-independence number. Hence, the tree-independence number of hyperbolic uniform disk graphs of radius $r$ is $\mathcal{O}(\min \{\sqrt{n}, (1 + 1/r)\log^2 n\})$. This result allows to nicely interpolate between the almost Euclidean case of $r \in \mathcal{O}(1/\sqrt{n})$, where we get $\mathcal{O}(\sqrt{n})$ tree-independence number, thus matching the tight bound for unit disk graphs, and the firmly hyperbolic case $r \in \Omega(\log n)$, where we get $\mathcal{O}(\log n)$ tree-independence number. Note also that our $\mathcal{O}(\frac{r}{\tanh r})$ bound for layered tree-independence number gives a $\mathcal{O}(\sqrt{\frac{r}{\tanh r}}\cdot \sqrt{n})$ bound for tree-independence number, thus recovering the $\mathcal{O}(\sqrt{n})$ bound for $r \in \mathcal{O}(1)$. 

Results \eqref{item:map} and \eqref{item:power}, pipelined with \cite[Lemma~3.2]{DMS24} (see \Cref{ramsey} below), immediately imply that $(g, d)$-map graphs have bounded layered treewidth \cite{DEW17} and that powers of bounded layered treewidth and bounded degree graphs have bounded layered treewidth \cite{DEMWW22}. However, as one should expect, the bounds obtained in this way are worse than those from \cite{DEW17,DEMWW22}. 

It might be tempting to believe that, whenever the $K_r$-free graphs in a class $\mathcal{G}$ have bounded layered treewidth (as is the case when $\mathcal{G}$ consists of $g$-map graphs \cite{DEW17} or of powers of bounded layered treewidth graphs \cite{DEMWW22}), then the class $\mathcal{G}$ itself has bounded layered tree-independence number. This, however, seems not to be the case, in view of recent constructions of Chudnovsky and Trotignon~\cite{CT24} of graph classes whose treewidth is bounded by a function of the clique number while the tree-independence number is unbounded.

Results \eqref{item:hyper} and \eqref{item:spher} show that, in a sense, the structural properties of unit disk graphs carry over both to surfaces of positive and negative curvature. 
The proofs of these two results follow the same strategy and differ from the Euclidean case addressed in \cite{GMY24}. This is somehow to be expected, given that the hyperbolic plane and the unit sphere cannot be isometrically embedded (as metric spaces) into any Euclidean space (see, e.g., \cite{Rob06}). Many equivalent models of the hyperbolic plane have been proposed \cite{RR} and, in the proof of Result \eqref{item:hyper}, we work with the so-called polar-coordinate model. Loosely speaking, we fix a point $o \in \mathbb{H}^2$ and a ray starting at $o$, and we identify every point in the hyperbolic plane by its hyperbolic distance from $o$ and its angle around $o$, relative to the chosen ray. As mentioned, the proof of Result \eqref{item:spher} adopts a similar strategy, where this time we use an analogue polar-coordinate model for $\mathbb{S}^2$. 

Given that $0$-map graphs (which coincide with contact graphs of disk homeomorphs in $\mathbb{R}^2$) have bounded layered tree-independence number, it is natural to ask whether the same holds for \emph{contact string graphs}, that is, contact graphs of simple curves in $\mathbb{R}^2$ (i.e., contact graphs of homeomorphs of the unit interval $[0,1]$ in $\mathbb{R}^2$). We answer negatively by showing that even contact graphs of segments in $\mathbb{R}^2$ have unbounded layered tree-independence number. Hence, there exist contact string graphs which are not $0$-map graphs. We complement this by showing the following separation result: There exist $0$-map graphs which are not contact string graphs. 

\textit{Clique cover degeneracy.} In \Cref{s:ccdeg}, we bound the clique cover degeneracy of the graph classes with bounded layered tree-independence number discussed above and of contact string graphs. Bl{\"{a}}sius et al.~\cite{BvdH24} showed that hyperbolic uniform disk graphs have clique cover degeneracy at most $3$. Extending a result of Ye and Borodin~\cite{YB12}, we show that $0$-map graphs have clique cover degeneracy at most $3$, which is tight. We also show that spherical uniform disk graphs and contact string graphs have clique cover degeneracy at most $6$ and $4$, respectively. 

\textit{Toward \Cref{questionltreealpha}.} In \Cref{s:towardsquestion}, we observe how combining boundedness of layered tree-independence number and of clique cover degeneracy implies the existence of clique-based separators of sublinear weight, thus answering \Cref{questionltreealpha} in the positive for the aforementioned graph classes. More generally, we observe that the following (incomparable) classes with bounded layered tree-independence number admit clique-based separators of sublinear weight: every class of bounded layered tree-independence number with a subquadratic $\theta$-binding function,
every class of bounded layered treewidth,
every class of bounded tree-independence number.

\textit{Independence degeneracy.} In \Cref{sec:indep_deg}, we consider \Cref{fractional tin-fragility bounded independence degeneracy}. Although we fall short of proving the conjecture, we observe that the following relaxation holds: Every fractionally $\tin$-fragile graph class $\mathcal{G}$ is  polynomially $(\dg,\omega)$-bounded, i.e., the degeneracy of every induced subgraph $H$ of a graph $G \in \mathcal{G}$ is bounded by a polynomial function of the clique number of $H$.  

\subsection{Algorithmic consequences}

In this section we highlight the main algorithmic consequences of our work. First, since $g$-map graphs have bounded layered tree-independence number, the machinery from \cite{GMY24} implies that \textsc{$(c, h, \psi)$-Max Weight Induced Subgraph} and \textsc{Max Weight Distance-$d$ Packing} admit PTASes on $g$-map graphs. We refer the reader to \cite{GMY24} for the definitions of these two meta-problems. 

Recall now that de Berg et al.~\cite{dBKMT23} (see also \cite{BBK20}) give $2^{\mathcal{O}(w_{\mathcal{G}}(n))}$-time algorithms for \textsc{Independent Set} and \textsc{Feedback Vertex Set} on every class $\mathcal{G}$ of geometric intersection graphs admitting polynomial-time computable clique-based separators of weight $w_{\mathcal{G}}(n)$. Combining this with the fact that $0$-map graphs admit clique-based separators of weight $\mathcal{O}(\sqrt{n})$ \cite{dBKMT23} and that the same holds for disk graphs \cite{BBK20}, results in $2^{\mathcal{O}(\sqrt{n})}$-time algorithms for these two classes. Analogous arguments give $2^{\mathcal{O}(n^{2/3}\log n)}$-time algorithms for pseudo-disk graphs \cite{dBKMT23}.

A possible drawback of the techniques from \cite{BBK20,dBKMT23} for classes with small
clique-based separators is that they seem ill-suited to deal with \textit{weighted} problems in general (an exception is given by \textsc{Max Weight Independent Set}, see \cite{BBK20}). In fact, de Berg and Kisfaludi{-}Bak~\cite{BK20} asked to determine the complexity of the weighted versions of problems falling in the framework of \cite{BBK20} when restricted to intersection graphs of (similarly-sized) fat objects in $\mathbb{R}^2$. Toward this, they showed that \textsc{Min Weight Dominating Set} cannot be solved in $2^{o(n)}$ time on unit ball graphs in $\mathbb{R}^3$, unless the ETH fails. Hence, the weighted versions could be considerably harder. 

We observe that the aforementioned algorithms of Lima et al.~\cite{LMMORS24}, combined with our bounds for tree-independence number from \Cref{summarytable}, allow us to obtain the following algorithms in the \textit{weighted} setting. 

\begin{corollary}\label{algos} \textsc{Max Weight Distance-$d$ Packing}, for fixed even $d\in\mathbb{N}$, and \textsc{Min Weight Feedback Vertex Set} admit algorithms with running time: 
\begin{itemize}
\item $2^{\mathcal{O}(\sqrt{n}\log n)}$, for $g$-map graphs, spherical uniform disk graphs, and intersection graphs of convex fat objects or similarly-sized fat objects in $\mathbb{R}^k$, for any $k\geq 2$;
\item $2^{\mathcal{O}(n^{2/3}\log n)}$, for contact segment graphs and, more generally, pseudo-disk graphs;
\item $\min \big\{2^{\mathcal{O}(\sqrt{n}\log n)}, 2^{\mathcal{O}((1 + \frac{1}{r})\log^3 n)}\big\}$, for hyperbolic uniform disk graphs with radius $r$, for any $r$.
\end{itemize}
\end{corollary}

\Cref{algos} thus allows us to extend, in a simple and uniform way, several subexponential-time algorithms from \cite{BBK20,dBKMT23} to the weighted setting. The running times obtained in this way either match (in the case of pseudo-disk graphs) or are a $\log n$ factor off (in the case of $0$-map graphs and intersection graphs of convex fat objects or similarly-sized fat objects in $\mathbb{R}^k$) those for the unweighted problems from \cite{BBK20,dBKMT23}. Note, however, that our results for map graphs are slightly more general, as they hold for every $g$-map graph with $g \geq 0$.

Consider now hyperbolic uniform disk graphs with radius $r$. \Cref{algos} shows that a large number of weighted problems, including \textsc{Min Weight Feedback Vertex Set}, admit quasi-polynomial-time algorithms on such graphs when $r \in \Omega(1/\mathsf{polylog}(n))$. Recall the following two results of Bl{\"{a}}sius et al.~\cite{BvdH24} for graphs in $\text{HUDG}(r)$: \textsc{Independent Set} admits an $n^{\mathcal{O}(1 + (1/r)\log n)}$-time algorithm, for any $r$, whereas \textsc{Feedback Vertex Set} admits an $n^{\mathcal{O}((1/r)\log n)}$-time algorithm, for $r \in \mathcal{O}(1)$. \Cref{algos} can then be viewed as extending the latter result and showing that both \textsc{Max Weight Independent Set} and \textsc{Min Weight Feedback Vertex Set} are quasi-polynomial-time solvable when $r \in \Omega(1/\mathsf{polylog}(n))$. It also partially answers a question of Bl{\"{a}}sius et al.~\cite{BvdH24}, who asked to investigate how the complexity of problems other than \textsc{Independent Set} on hyperbolic uniform disk graphs scales with the radius of the disks.

It seems reasonable that \Cref{algos} can be further extended to the meta-problem \textsc{$(c, \psi)$-Max Weight Induced Subgraph}\footnote{This is just \textsc{$(c, h, \psi)$-Max Weight Induced Subgraph} without any monotonicity constraint.} defined in \cite{LMMORS24}, where an algorithm with running time $n^{\mathcal{O}(R(8k+1,c+1))}$ is provided; here $k$ is the tree-independence number of the input graph, and $R(8k+1,c+1)$ is upper-bounded by a polynomial in $k$ of degree $c$. In order to obtain subexponential running time for the graph classes in \Cref{algos}, it would be enough to show that the above problem is solvable in time $n^{\mathcal{O}(k)}$.   

\section{Preliminaries}\label{s:prelim}

We let $\mathbb{N} = \{1,2,3,\ldots\}$. In this paper we
consider only finite simple graphs. For $r \geq 0$, the \textit{$r$-th neighborhood} of a vertex $v\in V(G)$, denoted by $N_G^r[v]$, is the set of vertices at distance at most $r$ in $G$ from $v$. Given a subset $U\subseteq V(G)$, we let $N_G^r[U] = \bigcup_{u\in U}N_G^r[u]$. For simplicity, we may sometimes identify a set $S$ of vertices of a graph $G$ with the corresponding induced subgraph $G[S]$. 

Given a nonnull graph $G$, we denote by $\dg(G)$ the \emph{degeneracy} of $G$, defined as the maximum, over all nonnull induced subgraphs $H$ of $G$, of the minimum degree of $H$.
The \emph{independence degeneracy} of $G$ is defined as the maximum, over all nonnull induced subgraphs $H$ of $G$, of the minimum, over all vertices $v\in V(H)$, of the independence number of the subgraph of $H$ induced by $N_H(v)$.
Equivalently, the independence degeneracy of $G$ is the smallest integer $k$ such that every nonnull induced subgraph of $G$ has a vertex that is not the center of an induced star with more than $k$ leaves.
Let us note that this parameter was dubbed \emph{inductive independence number} by Ye and Borodin~\cite{YB12}. The \textit{clique cover degeneracy} of $G$ is defined similarly to the independence degeneracy of $G$, except that the independence number is replaced by the clique cover number. 

Given a graph $G$, the independence number of $G$ is denoted by $\alpha(G)$, the clique number of $G$ is denoted by $\omega(G)$, the chromatic number of $G$ is denoted by $\chi(G)$, and the clique cover number of $G$ is denoted by $\theta(G)$. A graph class $\mathcal{G}$ is said to be \emph{$\chi$-bounded} if it admits a \emph{$\chi$-binding function}, that is, a function $f$ such that for every graph $G$ in the class and every induced subgraph $H$ of $G$, it holds that $\chi(H)\le f(\omega(H))$ (see~\cite{Gyarfas87,SS20}).
Furthermore, if there exists a polynomial function with this property, then the class is said to be \emph{polynomially $\chi$-bounded}. A graph class $\mathcal{G}$ is said to be \emph{$\theta$-bounded} if it admits a \emph{$\theta$-binding function}, that is, a function $f$ such that for every graph $G$ in the class and every induced subgraph $H$ of $G$, it holds that $\theta(H)\le f(\alpha(H))$ (see~\cite{Gyarfas87}). 
Observe that $\chi(H) \leq f(\omega(H))$ if and only if $\theta(\overline{H}) \leq f(\alpha(\overline{H}))$ and so a graph class $\mathcal{G}$ is $\chi$-bounded with $\chi$-binding function $f$ if and only if the complementary graph class $\{\overline{G}\colon G \in \mathcal{G}\}$ is $\theta$-bounded with $\theta$-binding function $f$.   

A restricted version of $\chi$-boundedness is the following.
We say that a graph class $\mathcal{G}$ is \emph{$(\dg,\omega)$-bounded} 
(resp., \emph{polynomially $(\dg,\omega)$-bounded}) if there exists a function $f$ (resp., a polynomial function) such that for every graph $G$ in the class and every induced subgraph $H$ of $G$, it holds that $\dg(H)\le f(\omega(H))$.
While not every $\chi$-bounded graph class is polynomially $\chi$-bounded \cite{BDW24}, every $(\dg,\omega)$-bounded graph class is polynomially $(\dg,\omega)$-bounded (see, e.g.,~\cite[Corollary 3.5]{du2024survey}). 

A \textit{tree decomposition} of a graph $G$ is a pair $\mathcal{T} = (T, \{X_t\}_{t\in V(T)})$, where $T$ is a tree whose every node $t$ is assigned a vertex subset $X_t \subseteq V(G)$, called \textit{bag}, such that the following conditions are satisfied: 
\begin{description}
\item[(T1)] Every vertex of $G$ belongs to at least one bag; 
\item[(T2)] For every $uv \in E(G)$, there exists a bag containing both $u$ and $v$; 
\item[(T3)] For every $u \in V(G)$, the subgraph $T_u$ of $T$ induced by $\{t \in V(T)\colon u \in X_t\}$ is connected. 
\end{description}
The \textit{width} of $\mathcal{T} = (T, \{X_t\}_{t\in V(T)})$ is the maximum value of $|X_t| - 1$ over all $t \in V(T)$. The \textit{treewidth} of a graph $G$, denoted $\tw(G)$, is the minimum width of a tree decomposition of $G$. The \textit{independence number} of $\mathcal{T}$, denoted $\alpha(\mathcal{T}$), is the quantity $\max_{t\in V(T)} \alpha(G[X_t])$. The \textit{tree-independence number} of a graph $G$, denoted $\tin(G)$, is the minimum independence number of a tree decomposition of $G$ (see \cite{DMS24}). For a tree decomposition $\mathcal{T} = (T, \{X_t\}_{t\in V(T)})$, we let $|\mathcal{T}| = |V(T)|$. For positive integers $p$ and $q$, the \textit{Ramsey number} $R(p,q)$ is the smallest integer $n_0$ such that every graph with at least $n_0$ vertices contains
either a clique of size $p$ or an independent set of size $q$. The following holds.

\begin{lemma}[Dallard et al.~\cite{DMS24}]\label{ramsey} Let $G$ be a graph and let $k\in \mathbb{N}$. If $\tin(G) \leq k$, then $\tw(G) \leq R(\omega(G) + 1, k + 1) - 2$.    
\end{lemma}

Let $\mathcal{T} = (T,\{X_t\}_{t \in V(T)})$ be a tree decomposition of a graph $G$.
The \emph{layered width} of $\mathcal{T}$ is the minimum integer $\ell$ such that there exists a layering $(V_0,V_1,\ldots)$ of $G$ such that for each bag $X_t$ and each layer $V_i$, it holds that $|X_t\cap V_i|\le \ell$.
Similarly, the \emph{layered independence number} of $\mathcal{T}$ is the minimum integer $\ell$ such that there exists a layering $(V_0,V_1,\ldots)$ of $G$ such that for each $t\in V(T)$ and each $i\ge 0$, it holds that $\alpha(G[X_t\cap V_i])\le \ell$.
The \emph{layered treewidth} (resp., layered tree-independence number) of a graph $G$ is defined as the minimum layered width (resp., layered independence number) of a tree decomposition of $G$.
The notions of layered treewidth and layered tree-independence number were introduced by Dujmovi{\'c} et al.~\cite{DMW17} and Galby et al.~\cite{GMY24}, respectively. We say that a tree decomposition $\mathcal{T} = (T,\{X_t\}_{t \in V(T)})$ and a layering $(V_0,V_1,\ldots)$ of a graph $G$ \textit{witness} layered treewidth (resp., layered tree-independence number) $\ell$ if, for each bag $X_t$ and layer $V_i$, we have that $|X_t\cap V_i| \leq \ell$ (resp., $\alpha(G[X_t \cap V_i]) \leq \ell$). 
The following holds.

\begin{lemma}[Galby et al.~\cite{GMY24}]\label{treealphabound} Let $k \in \mathbb{N}$ and let $G$ be a graph on $n$ vertices. Given a tree decomposition $\mathcal{T}$ of $G$ and a layering $(V_0,V_1,\ldots)$ of $G$ witnessing layered tree-independence number $k$, it is possible to compute, in $\mathsf{poly}(n,|\mathcal{T}|)$ time, a tree decomposition of $G$ with independence number at most $2\sqrt{kn}$. In particular, every $n$-vertex graph with layered tree-independence number $k$ has tree-independence number at most $2\sqrt{kn}$.
\end{lemma}

Fractional $\tw$-fragility was introduced by Dvo\v{r}{\'{a}}k~\cite{Dvo16}. The following equivalent definition is taken from \cite{GMY23}. Let $p$ be a width parameter in $\{\tw, \tin\}$. For $\beta \leq 1$, a \textit{$\beta$-general cover} of a graph $G$ is a multiset $\mathcal{C}$ of subsets of $V(G)$ such that each vertex belongs to at least $\beta|\mathcal{C}|$ elements of the cover. The \textit{$p$-width} of the cover is $\max_{C \in \mathcal{C}}p(G[C])$. For a parameter $p$, a graph class $\mathcal{G}$ is \textit{fractionally $p$-fragile} if there exists a function $f\colon\mathbb{N}\rightarrow\mathbb{N}$ such that, for every $r\in\mathbb{N}$, every $G \in \mathcal{G}$ has a $(1 - 1/r)$-general cover with $p$-width at most $f(r)$. A fractionally $p$-fragile class $\mathcal{G}$ is \textit{efficiently fractionally $p$-fragile} if there exists an algorithm that, for every $r\in\mathbb{N}$ and $G \in \mathcal{G}$, returns in $\mathsf{poly}(|V(G)|)$ time a $(1 - 1/r)$-general cover $\mathcal{C}$ of $G$ and, for each $C \in \mathcal{C}$, a tree decomposition of $G[C]$ of width (if $p = \tw$) or independence number (if $p = \tin$) at most $f(r)$, for some function $f\colon\mathbb{N}\rightarrow\mathbb{N}$.

We finally provide the short proof of \Cref{cliquesep_totreealpha}, which we restate for convenience.

{\renewcommand\footnote[1]{}\observ*}

\begin{proof} Suppose that the clique-based separators of graphs in $\mathcal{G}$ are $\beta$-balanced, for some fixed $\beta < 1$. Let $G \in \mathcal{G}$ be an $n$-vertex graph. Consider a separation $(A, B)$ of $G$ with separator $A\cap B$ as in the statement. By assumption, $\max\{|A\setminus B|, |B\setminus A|\} \leq \beta n$ and there is no edge between $|A\setminus B|$ and $|B\setminus A|$. Given any tree decompositions of $G[A\setminus B]$ and $G[B\setminus A]$, we can combine them into a tree decomposition of $G$ by adding the vertices in $A \cap B$ to all bags of both decompositions, and arbitrarily connecting the two decomposition trees with an edge. Hence, denoting by $g(n)$ the maximum tree-independence number of an $n$-vertex graph in $\mathcal{G}$, we obtain that $g(n) \leq g(\lfloor\beta n\rfloor) + \mathcal{O}(f(n))$. Now, the first assertion in the statement follows from a simple inductive argument, whereas the second assertion follows from the Discrete master theorem (see, e.g., \cite{KL21}).
\end{proof}

\section{Bounded layered tree-independence number}\label{s:layeredtreealpha}

In this section we identify several graph classes with bounded layered tree-independence number.
We begin, in \Cref{neighbclique}, by introducing a graph transformation called $r$-neighborhood cliquification and show that $r$-neighborhood cliquifications of graphs of bounded layered treewidth result in graphs of bounded layered tree-independence number. We then provide two examples of well-known graph classes that can be
obtained by $r$-neighborhood cliquifications of graphs of bounded layered treewidth: powers of bounded layered treewidth graphs (\Cref{powersbounds}) and $g$-map graphs (\Cref{mapbounds}). However, the bounds for layered tree-independence number thus obtained are not optimal and we improve them in the corresponding sections. 

In \Cref{sec:hyperbolic,sec:spherical}, we finally bound layered tree-independence number for hyperbolic and spherical uniform disk graphs, respectively, two generalizations of unit disk graphs.

\subsection{$r$-neighborhood cliquifications}\label{neighbclique}

Given a graph $G$ and a subset $P \subseteq V(G)$, by \textit{cliquifying $P$} we denote the operation that adds all missing edges between vertices in $P$, resulting in a new graph. Given a graph $G$, a subset $P \subseteq V(G)$ and an integer $r \geq 1$, the \textit{$r$-neighborhood cliquification on $(G,P)$} is the operation
that consists in cliquifying $N^r_G[p]$, for each $p \in P$. More precisely, the resulting graph $H$ is the graph with $V(H)  = V(G)$ and such that, for distinct $u,v \in V(H)$, $uv\in E(H)$ if and only if $uv\in E(G)$ or there exists $p\in P$ such that $\{u,v\} \subseteq N^r_G[p]$. In the following, a $1$-neighborhood cliquification will be simply referred to as a \textit{neighborhood cliquification}.

\begin{theorem}\label{from-bdd-ltw-to-bdd-lta}
Let $G$ be a graph, let $P \subseteq V(G)$, let $r$ be a positive integer, and let $G'$ be the graph obtained by the $r$-neighborhood cliquification on $(G,P)$. 
Given a tree decomposition $\mathcal{T}$ and a layering $(V_1, V_2, \ldots)$ of $G$ witnessing layered treewidth $k$, it is possible to compute, in time \hbox{$\mathcal{O}(|\mathcal{T}|\cdot (|V(G)|+|E(G)|))$}, a tree decomposition $\mathcal{T}'$ with $|\mathcal{T}'| = |\mathcal{T}|$ and a layering of $G'$ witnessing layered tree-independence number $4rk$.

In addition, if $r = 1$ and $P$ is an independent set in $G$, the same conclusion holds even if $\mathcal{T}$ and $(V_1, V_2, \ldots)$ witness layered tree-independence number $k$ of $G$.    
\end{theorem}

\begin{proof}
Let $\mathcal{T} = (T,\{X_t\}_{t\in V(T)}\})$ be a tree decomposition of $G$ and let $(V_1, V_2, \ldots)$ be a layering of $G$ witnessing layered treewidth $k$ (or layered tree-independence number $k$, if $r=1$ and $P$ is an independent set).
We first show how to construct a layering of $G'$. For each $i \geq 1$, let $W_i = \bigcup_{j = 1}^{2r} V_{2(i-1)r+j}$.

\begin{nitem}\label{contactgraphtheoremlayering-new}
    $(W_1, W_2, \ldots)$ is a layering of $G'$.
\end{nitem}

\begin{claimproof}[Proof of \eqref{contactgraphtheoremlayering-new}]
Clearly, $(W_1, W_2, \ldots)$ is a partition of $V(G')$. Let $uv \in E(G')$. We distinguish two cases. Suppose first that $uv \in E(G)$. Then $u$ and $v$ belong to layers from $(V_1, V_2, \ldots)$ that are at distance at most $1$, and so the same holds for the layers from $(W_1, W_2, \ldots)$ they belong to. 
Suppose now that $uv \not \in E(G)$. 
Since $uv \in E(G')$, there exists $p\in P$ such that $\{u,v\} \subseteq N_{G}^r[p]$. 
This implies that $u$ and $p$ belong to layers from $(V_1, V_2, \ldots)$ that are at distance at most $r$, and the same holds for $v$ and $p$. 
Hence, $u$ and $v$ belong to layers from $(V_1, V_2, \ldots)$ that are at distance at most $2r$, and so the layers from $(W_1, W_2, \ldots)$ they belong to are at distance at most $1$.  
\end{claimproof}
 
We now build a tree decomposition $\mathcal{T}' = (T,\{Y_t\}_{t\in V(T)})$ of $G'$ as follows. 
For each $u\in V(G)$ and $p\in P$ such that $u \in N^r_G[p]$, fix a shortest $u, p$-path in $G$. We denote the set of vertices of such a path as $V(u,p)$. 
For each $t \in V(T)$, we define $Y_t$ as the set of all vertices $u \in V(G')$ such that either $u \in X_t$ or there exist $v\in X_t$ and $p\in P$ with $v\in V(u,p)$. 
In other words, letting $Y_t^- = X_t$ and $Y_t^+ = {\{u\in V(G')\colon  V(u,p) \cap X_t\neq\varnothing \text{ for some } p\in P\}}$, we have that $Y_t = Y_t^-\cup Y_t^+$.
Note that, for every $t\in V(T)$, the set $Y_t^+$ is a subset of $N_G^r[P]$.

\begin{nitem}\label{contactgraphtheoremtreedecomposition-new}
    $\mathcal{T}'$ is a tree decomposition of $G'$.
\end{nitem}
    
\begin{claimproof}[Proof of \eqref{contactgraphtheoremtreedecomposition-new}]
        Consider first (T1). Let $v\in V(G')$. Since $\mathcal{T}$ is a tree decomposition of $G$, there exists $t\in V(T)$ such that $v \in X_t$. Then $v\in Y_t^-$ and so $v\in Y_t$.
       
        Consider now (T2). Let $u$ and $v$ be two adjacent vertices in $G'$.  
        Suppose first that $uv \in E(G)$. 
        Since $\mathcal{T}$ is a tree decomposition of $G$, there exists $t \in V(T)$ such that $\{u,v\} \subseteq X_t$, and so $\{u,v\} \subseteq Y_t^- \subseteq Y_t$. 
        Suppose next that $uv \not \in E(G)$. Since $u$ and $v$ are adjacent in $G'$, there exists $p\in P$ such that $\{u,v\} \subseteq N_{G}^r[p]$. 
        But there exists $t\in V(T)$ such that $p\in X_t$, and so the intersections $V(u,p) \cap X_t$ and $V(v,p) \cap X_t$ are both nonempty, since they contain $p$, implying that $\{u,v\}\subseteq Y^+_t$ and so $\{u,v\}\subseteq Y_t$. 
        
        Consider finally (T3). 
        For each vertex $v \in V(G)$, let $T_v$ and $T'_v$ denote the subgraphs of $T$ induced by $\{t \in V(T)\colon v \in X_t\}$ and $\{t \in V(T)\colon v \in Y_t\}$, respectively. 
        Let $u$ be an arbitrary vertex of $G'$.
        Suppose first that $u \not \in N^r_G[p]$ for every $p\in P$. Hence, $u$ does not belong to any set $Y_t^+$ with $t\in V(T)$ (since each such set is a subset of $N_G^r[P]$), and so the condition $u\in Y_t$ is equivalent to the condition $u\in X_t$, implying that $T'_u = T_u$ is a tree. 
        
        Suppose finally that $u \in N^r_G[p]$ for at least one vertex $p\in P$ and let $P^{u} = \{p \in P : u \in N^r_G[p]\}$.
        Observe first that $V(T'_u) = \bigcup_{p\in P^{u}}\bigcup_{v\in V(u,p)}V(T_v)$. 
        We now show that, for every $p \in P^{u}$, the set $\bigcup_{v\in V(u,p)}V(T_v)$ induces a tree. To this end, let $V(u,p) = \{v_1,\ldots,v_k\}$ and recall that $V(u,p)$ induces a $u, p$-path in $G$. Suppose that $v_1=u$, $v_k=p$, and $v_i$ is adjacent to $v_{i-1}$ for every $i\in \{2,\ldots, k\}$. 
        Let $V_i^u = \bigcup_{j=1}^{i}V(T_{v_j})$. 
        We claim that $V_i^u$ induces a tree for every $i\in \{1,\ldots, k\}$.
        We proceed by induction on $i$.
        The base case $i=1$ follows from the fact that $V_1^u = V(T_u)$ induces a tree. 
        Assume now inductively that $V_i^u$ induces a tree.
        Consider $V_{i+1}^u$, which is the union of $V_i^u$ and $V(T_{v_{i+1}})$, both of which induce trees. 
        Since $v_{i+1}$ is adjacent to $v_{i}$ in $G$, there is a node of $T$ whose bag $X_t$ contains $\{v_{i+1}, v_i\}$, and this node belongs to both $V_i^u$ and $V(T_{v_{i+1}})$.
        Hence, $V_{i+1}^u$ induces a connected subgraph, completing the induction step.    Letting $i = k$ in the previous claim, we obtain that $\bigcup_{v\in V(u,p)}V(T_v)$ induces a tree for every $p \in P^{u}$. Since all these trees contain $V(T_u)$, we conclude that $V(T'_u) = \bigcup_{p\in P^{u}}\bigcup_{v\in V(u,p)}V(T_v)$ indeed induces a tree. 
        \end{claimproof}

Note that the computation of the layering $(W_1,W_2,\ldots)$ and of the tree decomposition $\mathcal{T}'$ of $G'$ can be done in $\mathcal{O}(|\mathcal{T}|\cdot (|V(G)|+|E(G)|))$ time by performing a breadth-first search up to distance $r$ in $G$ from each bag $X_t$.

It remains to show that $\alpha(G'[Y_t \cap W_i]) \leq 4rk$ for each bag $Y_t$ and layer $W_i$. To this end, consider an arbitrary bag $Y_t$ and layer $W_i$ and let $A' = \{a_1', a_2', \ldots\}$ be an independent set of $G'[Y_t \cap W_i]$. For each $a_j' \in A'$, we do the following.
If $a_j' \in Y_t^-$, then let $a_j = a_j'$. 
Otherwise, $a_j' \in Y_t^+ \setminus Y_t^-$, and so there exists $p\in P$ and $v\in X_t$ such that $v\in V(a_j',p)$. In this case, let $a_j = v$, where $v$ is an arbitrary such vertex. Note that $v \neq a_j'$, since otherwise $a_j'\in Y_t^-$. Finally, let $A = \{a_1, a_2, \ldots\}$. 
By construction, $A \subseteq X_t$ and $|A|\le |A'|$.

\begin{nitem}\label{contactgraphtheoremb}
    $|A| = |A'|$. Moreover, if $r=1$ and $P$ is an independent set in $G$, then $A$ is an independent set in $G$.
\end{nitem}
   
\begin{claimproof}[Proof of \eqref{contactgraphtheoremb}] Consider the first assertion.
Take an arbitrary pair $a_j', a_\ell' \in A'$ with $a_j'\neq a_\ell'$. 
We aim to show that $a_j \neq a_\ell$. 
This clearly holds if $a_j = a_j'$ and $a_\ell = a_\ell'$. 
Hence, we may assume without loss of generality that $a_j \neq a_j'$. 
This means that there exists $p_j \in P$ such that $a_j \in V(a_j',p_j)$.
Suppose, to the contrary, that $a_j = a_\ell$. Then, $a_\ell \in N^r_{G}[p_j]$.
If $a_\ell = a_\ell'$, then $\{a_\ell',a_j'\} \subseteq N^r_G[p_j]$, and so $a_\ell'$ and $a_j'$ are adjacent in $G'$, contradicting the fact that $A'$ is an independent set of $G'$.
Hence, $a_\ell \neq a_\ell'$ and so there exists $p_\ell \in P$ such that $a_\ell \in V(a_\ell',p_\ell)$. Observe that $p_\ell \neq p_j$, or else $\{a_\ell',a_j'\} \subseteq N^r_G[p_\ell]$ and $a_\ell'$ and $a_j'$ are adjacent in $G'$, a contradiction.
Now, since $a_j \in V(a_j',p_j)$ and $a_\ell \in V(a_\ell',p_\ell)$, we have that $d_G(a_j',a_j) + d_G(a_j,p_j) \leq r$ and $d_G(a_\ell',a_\ell) + d_G(a_\ell,p_\ell) \leq r$.
Therefore, since $a_j = a_\ell$, either $d_G(a_j',a_j) + d_G(a_j,p_\ell) \leq r$ or $d_G(a_\ell',a_j) + d_G(a_j,p_j) \leq r$.
If the former holds, then $a_j' \in N_G^r[p_\ell]$, so $a_j'$ and $a_\ell'$ are adjacent in $G'$, a contradiction.
If the latter holds, then $a_\ell' \in N_G^r[p_j]$, so $a_\ell'$ and $a_j'$ are adjacent in $G'$, a contradiction again. This shows that $a_j \neq a_\ell$ and so $|A| = |A'|$.

Consider now the second assertion. Let $r=1$ and let $P$ be an independent set in $G$. Observe that, since $r=1$, $V(a_i', p) = \{a_i', p\}$ for every $a_i' \in A'$ and $p \in P$. Suppose, to the contrary, that there exist $a_j, a_\ell \in A$ with $a_j a_\ell \in E(G)$. 
Since $A'$ is an independent set of $G'$, we cannot have both $a_j' = a_j$ and $a_\ell = a_\ell'$. Moreover, since $P$ is an independent set in $G$, one of $a_j$ and $a_\ell$ does not belong to $P$, say without loss of generality $a_\ell \notin P$. But then $a_\ell = a_\ell'$, or else there exists $p \in P$ such that $a_\ell = v$ for some $v \in V(a_\ell',p) = \{a_\ell', p\}$, a contradiction.
Since $a_\ell = a_\ell'$, the previous observation implies that $a_j \neq a_j'$, from which we obtain that $a_j \in P$ and $a_j' \in N_G(a_j)$.
Since $a_ja_\ell \in E(G)$, we conclude that $\{a_\ell', a_j '\} \subseteq N_G(a_j)$, so $a_\ell'$ and $a_j'$ are adjacent in $G'$, a contradiction.\end{claimproof}

\begin{nitem}\label{contactgraphtheoremc} Each $a_j \in A$ belongs to $\bigcup_{s = 1-r}^{3r} V_{2(i-1)r+s}$.
\end{nitem}

\begin{claimproof}[Proof of \eqref{contactgraphtheoremc}] 
Let $a_j \in A$. If $a_j' \in Y_t^-$, then $a_j = a_j' \in W_i \subseteq \bigcup_{s = 1}^{2r} V_{2(i-1)r+s}$.
Otherwise, $a_j' \in Y_t^+ \setminus Y_t^-$, and so $a_j \in V(a_j',p)$ for some $p\in P$. Since $V(a_j',p)$ induces a shortest $a_j', p$-path in $G$ and $d_G(a_j',p) \leq r$, it must be that $d_G(a_j',a_j) \leq r$. Hence, $a_j$ and $a_j'$ belong to layers from $(V_1, V_2, \ldots)$ that are at distance at most $r$. Since $a_j' \in  \bigcup_{s = 1}^{2r} V_{2(i-1)r+s}$, we conclude that $a_j \in \bigcup_{s = 1-r}^{3r} V_{2(i-1)r+s}$.
\end{claimproof}

We can finally show the two assertions in the statement. 
Suppose, to the contrary, that $\alpha(G'[Y_t \cap W_i])\ge 4rk+1$. 
By \eqref{contactgraphtheoremb}, \eqref{contactgraphtheoremc}, and the pigeonhole principle, there exists a layer $V_j$ containing at least $k+1$ vertices of $A$. 
This implies that $|X_t \cap V_j| \geq k+1$ and, if in addition $r=1$ and $P$ is an independent set in $G$, that $\alpha(G[X_t \cap V_j]) \geq k+1$, contradicting the corresponding assumptions.
\end{proof}

\subsection{Powers of bounded layered treewidth graphs}\label{powersbounds}

\Cref{from-bdd-ltw-to-bdd-lta} can be used to show that powers of bounded layered treewidth graphs have bounded layered tree-independence number. The argument is as follows. Let $G$ be a graph with layered treewidth $k$. Consider first the case of even powers of $G$. Add a pendant vertex to each $v \in V(G)$ in order to obtain a graph $G'$ and let $P = V(G')\setminus V(G)$ be the set of newly added vertices. Observe that, for $r \geq 2$, the graph obtained by the $r$-neighborhood cliquification on $(G', P)$ followed by deleting $P$ is isomorphic to $G^{2(r-1)}$. Moreover, $G'$ has layered treewidth at most $2k$. To see this, take a tree decomposition $\mathcal{T}$ and a layering $(V_0, V_1, \ldots)$ of $G$ witnessing layered treewidth $k$. For each $u \in V(G')\setminus V(G)$, add $u$ to every bag of $\mathcal{T}$ its unique neighbor belongs to, and add $u$ to the layer in $(V_0, V_1, \ldots)$ its unique neighbor belongs to. This gives a tree decomposition and a layering of $G'$ witnessing layered treewidth $2k$. But then \Cref{from-bdd-ltw-to-bdd-lta} implies that $G^{2(r-1)}$  has layered tree-independence number at most $8rk$. 

Consider now the case of odd powers of $G$. Build $G'$ from $G$ by adding, for every $e \in E(G)$, a new vertex adjacent precisely to the endpoints of $e$ and let $P = V(G')\setminus V(G)$. Observe that the graph obtained by the $r$-neighborhood cliquification on $(G', P)$ followed by deleting $P$ is isomorphic to $G^{2r-1}$. Similarly to the previous paragraph, it is easy to see that $G'$ has layered treewidth at most $k + \binom{k}{2}$. But then \Cref{from-bdd-ltw-to-bdd-lta} implies that $G^{2r-1}$ has layered tree-independence number at most $4r(k + \binom{k}{2})$.

In the following, we improve on these bounds by directly constructing a tree decomposition and a layering of a power of $G$. The construction we give is similar to the one used in the proof of \cite[Lemma 7]{DEMWW22} showing that powers of graphs with bounded layered treewidth and bounded degree have bounded layered treewidth. 

\begin{theorem}\label{powers}
Let $G$ be a graph and let $d\ge 2$ be an integer.
Given a tree decomposition $\mathcal{T}$ and a layering $(V_1, V_2, \ldots)$ of $G$ witnessing layered treewidth $k$, it is possible to compute, in time $\mathcal{O}(|\mathcal{T}|\cdot (|V(G)|+|E(G)|))$, a tree decomposition $\mathcal{T}^d$ with $|\mathcal{T}^d| = |\mathcal{T}|$ and a layering of the graph $G^{d}$ witnessing layered tree-independence number $2dk$ if $d$ is even, and $(2d-1)k$ if $d$ is odd.
\end{theorem}

\begin{proof}
Let $\mathcal{T} = (T,\{X_t\}_{t\in V(T)}\})$ be a tree decomposition of $G$ and let $(V_1, V_2, \ldots)$ be a layering of $G$ witnessing layered treewidth $k$.
We first show how to construct a layering of $G^d$. 
For each $i \geq 1$, let $W_i = \bigcup_{j = 1}^dV_{(i-1)d+j}$, that is, the sets $W_i$ are obtained by merging $d$ consecutive layers of $G$.

\begin{nitem}\label{powerlayering}
    $(W_1, W_2, \ldots)$ is a layering of $G^d$.
\end{nitem}

\begin{claimproof}[Proof of \eqref{powerlayering}]
Clearly, $(W_1, W_2, \ldots)$ is a partition of $V(G) = V(G^d)$. 
Let $uv \in E(G^d)$. 
Since there exists a $u,v$-path of length at most $d$ in $G$, the vertices $u$ and $v$ belong to layers from $(V_1, V_2, \ldots)$ that are at distance at most $d$.
Hence, $u$ and $v$ belong to layers from $(W_1, W_2, \ldots)$ that are at distance at most $1$.
\end{claimproof}
We now build a tree decomposition $\mathcal{T}^d = (T,\{Z_t\}_{t\in V(T)})$ of $G^d$ as follows. 
For each $t \in V(T)$, we define $Z_t$ to be the set of all vertices of $G^d$ that are at distance at most $\lfloor d/2\rfloor$ from $X_t$ in $G$.

\begin{nitem}\label{powerstreedecomposition}
    $\mathcal{T}^d$ is a tree decomposition of $G^d$.
\end{nitem}

\begin{claimproof}[Proof of \eqref{powerstreedecomposition}]
        Consider first (T1). 
        Let $v\in V(G^d)$. 
        Since $\mathcal{T}$ is a tree decomposition of $G$, there exists $t\in V(T)$ such that $v \in X_t$. 
        Hence, $v\in Z_t$.
        
        Consider now (T2). 
        Let $u$ and $v$ be two adjacent vertices in $G^d$.  
        The distance between $u$ and $v$ in $G$ is at least $1$ and at most $d$.
        Let $P = v_0\cdots v_q$ be a shortest $u,v$-path in $G$ with $u = v_0$ and $v = v_q$. Clearly, $1\le q\le d$.
        Let now $j = \lfloor q/2\rfloor$ and note that $j\le \lfloor d/2\rfloor$.
        We distinguish two cases according to the parity of $q$.
        Suppose first that $q$ is even.
        Then, $j = q/2$.
        Since $\mathcal{T}$ is a tree decomposition of $G$, there exists $t\in V(T)$ such that $v_j \in X_t$. 
        Consequently, since $u = v_0$ and $v = v_q$ are both at distance at most $j \le \lfloor d/2\rfloor$ from $v_j$ in $G$, we have that $\{u,v\}\subseteq Z_t$.
        Suppose now that $q$ is odd.
        Then, $j = (q-1)/2$.
        Consider the edge $v_jv_{j+1}$.
        Since $\mathcal{T}$ is a tree decomposition of $G$, there exists $t\in V(T)$ such that $\{v_j,v_{j+1}\} \subseteq X_t$. 
        Consequently, since $u = v_0$ and $v = v_q$ are both at distance at most $j\le \lfloor d/2\rfloor$ from $\{v_j,v_{j+1}\}$ in $G$, we have that $\{u,v\}\subseteq Z_t$.
         
        Consider finally (T3). 
        For each vertex $v \in V(G)$, let $T_v$ and $T^d_v$ denote the subgraphs of $T$ induced by $\{t \in V(T)\colon v \in X_t\}$ and $\{t \in V(T)\colon v \in Z_t\}$, respectively. 
        Let $u$ be an arbitrary vertex of $G^d$. 
        Recall that if $u \in Z_t$, then there exists a vertex $v \in X_t$ such that the distance between $u$ and $v$ in $G$ is at most $\lfloor d/2\rfloor$.    
        Hence, $T^d_u$ is the union of the tree $T_u$ and all the trees $T_v$ such that the distance between $u$ and $v$ in $G$ is at most $\lfloor d/2\rfloor$.
        Using the fact that $\mathcal{T}$ is a tree decomposition of $G$ and induction on $j$, it can be shown that for all $j\ge 1$, the union of the tree $T_u$ and all the trees $T_v$ such that the distance between $u$ and $v$ in $G$ is at most $j$, is connected.
        Hence, in particular, this is the case for $T^d_u$.
\end{claimproof}

Note that the computation of the layering $(W_1,W_2,\ldots)$ and of the tree decomposition $\mathcal{T}^d$ of $G^d$ can be done in the stated time, $\mathcal{O}(|\mathcal{T}|\cdot (|V(G)|+|E(G)|))$, by performing a breadth-first search up to distance $\lfloor d/2\rfloor$ in $G$ from each bag $X_t$.

It remains to show that, for each bag $Z_t$ and layer $W_i$, it holds that $\alpha(G^d[Z_t \cap W_i]) \leq 2dk$ if $d$ is even, and that $\alpha(G^d[Z_t \cap W_i]) \leq (2d-1)k$ if $d$ is odd.
We do this by showing that the set $Z_t \cap W_i$ is a union of $2dk$ or $(2d-1)k$ cliques in $G^d$, respectively.
Consider such a set $Z_t \cap W_i$.
Recall that the set $W_i$ is a union of $d$ consecutive layers of $G$.
Furthermore, the definition of $Z_t$ implies that each vertex $u\in Z_t \cap W_i$ is at distance at most $\lfloor d/2\rfloor$ in $G$ from some vertex $v\in X_t$.
Since $(V_1,V_2,\ldots)$ is a layering of $G$, any two such vertices $u$ and $v$ belong to layers at distance at most $\lfloor d/2\rfloor$.
Consequently, there exists a set $J$ of consecutive positive integers such that $|J|\le d+2\lfloor d/2\rfloor$ and $Z_t\cap W_i$ is a subset of the set of vertices at distance at most $\lfloor d/2\rfloor$ in $G$ from  the set $\bigcup_{j\in J}(X_t\cap V_j)$.
Note that for each vertex $v\in X_t$, the set of vertices at distance at most $\lfloor d/2\rfloor$ from $v$ in $G$ is a clique in $G^d$.
Since the set $\bigcup_{j\in J}(X_t\cap V_j)$ has cardinality at most $|J|k$, it follows that $Z_t\cap W_i$ is a union of at most $|J|k$ cliques in $G^d$.
To complete the proof, it suffices to observe that $|J|\le d+2\lfloor d/2\rfloor$ is bounded from above by $2d$ if $d$ is even, and by $2d-1$ if $d$ is odd.
\end{proof}

\begin{remark}\label{layeredcliquepower} The proof of \Cref{powers} actually shows something stronger than boundedness of layered tree-independence number. Indeed, each intersection between a bag and a layer not only has small independence number but also small clique cover number.
\end{remark}

Note that \Cref{powers} together with \Cref{ramsey} imply that, for every class of graphs with bounded maximum degree and bounded layered treewidth, fixed powers of graphs in the class also have bounded layered treewidth, as proved in \cite[Lemma~7]{DEMWW22} (with a better bound than what follows from our approach). 

\subsection{Map graphs and contact string graphs}\label{mapbounds}

A second application of \Cref{from-bdd-ltw-to-bdd-lta} concerns the class of $g$-map graphs, which can be characterized as follows.
Consider a bipartite graph $H$ with a bipartition $(X,Y)$.
The \emph{half-square graph} $H^2[X]$ is the graph with vertex set $X$, where two vertices are adjacent if and only if they are adjacent in $H$ to a common vertex in $Y$.
Then a graph $G$ is a $g$-map graph if and only if there exists a bipartite graph $H$ of Euler genus at most $g$ with a bipartition $(X,Y)$ such that $G$ is the half-square graph $H^2[X]$ (see~\cite[Lemma 5.1]{DEW17}), and $H$ is called a \textit{witness} of $G$.
Note that in this case $G$ is the result of a neighborhood cliquification on $(H,Y)$ followed by deleting $Y$. Moreover, graphs of bounded Euler genus have bounded layered treewidth, as implied by the following result that can be deduced from the proof of \cite[Theorem~12]{DMW17} (see also \cite[Lemma~5.3]{DEW17}).

\begin{theorem}[Dujmovi\'{c} et al.~\cite{DMW17}]\label{layeredtwplanar} 
Let $g \geq 0$. 
Given a connected $n$-vertex graph $G$ of Euler genus at most $g$, it is possible to compute, in $\mathsf{poly}(n,g)$ time, a tree decomposition $\mathcal{T}$ with $|\mathcal{T}| =\mathcal{O}(n)$ and a BFS layering $(V_1, V_2, \ldots)$ of $G$ witnessing layered treewidth $2g+3$.
\end{theorem}

Therefore, combining \Cref{layeredtwplanar} with \Cref{from-bdd-ltw-to-bdd-lta} immediately implies that $g$-map graphs have layered tree-independence number at most $8g + 12$. We now improve on this bound. 

Given a bipartite graph $G$ with bipartition $(B_1,B_2)$, a layering $(V_1,V_2,\ldots)$ is a \textit{bipartite layering} if, for each $i\in \mathbb{N}$, $V_{2i-1} \subseteq B_1$ and $V_{2i} \subseteq B_2$. 
The next result uses ideas from \Cref{from-bdd-ltw-to-bdd-lta} and \cite[Lemma~5.2]{DEW17}. 

\begin{proposition}\label{bipartite_layering}
Let $G$ be a bipartite graph with bipartition $(B_1,B_2)$ and let $G'$ be the graph obtained by the neighborhood cliquification on $(G, B_2)$ followed by deleting $B_2$. Given a tree decomposition $\mathcal{T}$ and a bipartite layering $(V_1, V_2, \ldots)$ of $G$ witnessing layered treewidth $k$, it is possible to compute, in time $\mathcal{O}(|\mathcal{T}|\cdot |V(G)|)$, a tree decomposition $\mathcal{T}'$ with $|\mathcal{T}'| = |\mathcal{T}|$ and a layering of $G'$ witnessing layered tree-independence number $3k$.
\end{proposition}

\begin{proof}
Let $\mathcal{T} = (T,\{X_t\}_{t\in V(T)})$ be a tree decomposition of $G$ and let $(V_1,V_2,\ldots)$ be a bipartite layering of the bipartite graph $G$ with bipartition $(B_1,B_2)$ such that, for each bag $X_t$ and layer $V_i$, $|X_t \cap V_i| \leq k$.  
We first construct a layering of $G'$. For each $i \geq 1$, let $W_i = V_{2i-1}$.

\begin{nitem}\label{biplay}
    $(W_1, W_2, \ldots)$ is a layering of $G'$.
\end{nitem}

\begin{claimproof}[Proof of \eqref{biplay}]
Clearly, $(W_1, W_2, \ldots)$ is a partition of $V(G')$. 
Note that $V(G') = B_1$, hence, for any pair of adjacent vertices $u$ and $v$ in $G'$, there exists $p\in B_2$ such that $\{u,v\} \subseteq N_{G}(p)$. 
This implies that $u$ and $p$ belong to layers from $(V_1, V_2, \ldots)$ that are at distance at most $1$, and the same holds for $v$ and $p$. 
Hence, $u$ and $v$ belong to layers from $(V_1, V_2, \ldots)$ that are at distance at most $2$, and so the layers from $(W_1, W_2, \ldots)$ they belong to are at distance at most $1$.  
\end{claimproof}

We now build a tree decomposition $\mathcal{T}' = (T,\{Y_t\}_{t\in V(T)})$ of $G'$ exactly as in the proof of \Cref{from-bdd-ltw-to-bdd-lta} with $r=1$. Observe that, for each $t \in V(T)$, $Y_t$ is the set of all vertices of $G'$ that either belong to $X_t$ or have a neighbor $p\in B_2$ (in $G$) that belongs to $X_t$. 
 
It remains to show that $\alpha(G'[Y_t \cap W_i]) \leq 3k$ for each bag $Y_t$ and layer $W_i$. 
Consider an arbitrary bag $Y_t$ and layer $W_i$.
Suppose for a contradiction that $\alpha(G'[Y_t \cap W_i])\ge 3k+1$ and let $A' = \{a_1', a_2', \ldots\}$ be an independent set in $G'[Y_t \cap W_i]$ such that $|A'|=3k+1$.
Taking $P = B_2$ in the proof of \Cref{from-bdd-ltw-to-bdd-lta}, construct the set $A = \{a_1, a_2, \ldots\} \subseteq X_t$ with $|A| = |A'|$. 

\begin{nitem}\label{biplay_1}
    Each $a_j \in A$ belongs to $V_{2i-2}\cup V_{2i-1} \cup V_{2i}$.
\end{nitem}

\begin{claimproof}[Proof of \eqref{biplay_1}] 
Let $a_j \in A$. If $a_j' \in Y_t^-$, then $a_j = a_j' \in W_i = V_{2i-1}$.
Otherwise, $a_j' \in Y_t^+ \setminus Y_t^-$, and so $a_j$ belongs to $B_2$ and is adjacent to $a_j' \in W_i = V_{2i-1}$. 
This implies that $a_j \in V_{2i-2}\cup V_{2i-1} \cup V_{2i}$.
\end{claimproof}

Since $|A| = |A'|=3k+1$, there exists a layer $V_j$ containing at least $k+1$ vertices of $A$. 
This implies that $|X_t \cap V_j| \geq k+1$, a contradiction.
\end{proof}

We now apply \Cref{bipartite_layering} to the class of $g$-map graphs. In the following result we assume, without loss of generality,\footnote{If instead $G$ is given with an embedding, a witness $H$ of $G$ can be computed in linear time thanks to \cite[Lemma 5.1]{DEW17}.} that every $g$-map graph $G$ is given with a witness $H$. 

\begin{corollary}\label{maplayered} Let $g\ge 0$. Given a $g$-map graph $G$ with $n$ vertices, it is possible to compute, in $\mathsf{poly}(n,g)$ time, a tree decomposition $\mathcal{T}$ with $|\mathcal{T}| =\mathcal{O}(n)$ and a layering $(V_1, V_2, \ldots)$ of $G$ witnessing layered tree-independence number $6g+9$. In particular, every $g$-map graph has layered tree-independence number at most $6g+9$.
\end{corollary}

\begin{proof} Let $G$ be a $g$-map graph with $n$ vertices. Let $H$ be the witness of $G$ with bipartition $(X, Y)$, and recall that $H$ has Euler genus at most $g$ and $G$ is obtained from a neighborhood cliquification on $(H,Y)$ followed by deleting $Y$. We apply \Cref{layeredtwplanar} to the graph $H$ in order to compute, in $\mathsf{poly}(n,g)$ time, a tree decomposition $\mathcal{T}$ with $|\mathcal{T}| =\mathcal{O}(n)$ and a BFS layering $(V_1, V_2, \ldots)$ of $H$ witnessing layered treewidth $2g+3$. Observe that $(V_1, V_2, \ldots)$ must be a bipartite layering of $H$. We then apply \Cref{bipartite_layering} in order to compute the desired tree decomposition and layering of $G$.
\end{proof}

\Cref{maplayered} and \Cref{treealphabound} imply the following.

\begin{corollary}\label{treealphamap}
    Every $g$-map graph on $n$ vertices has tree-independence number $\mathcal{O}(\sqrt{gn})$. 
    A tree decomposition witnessing this can be computed in $\mathsf{poly}(n,g)$ time.
\end{corollary}

\begin{remark} Note that \Cref{maplayered} combined with \Cref{ramsey} give an upper bound of $R(6g+10,d)-1$ for the layered treewidth of $(g, d)$-map graphs. Dujmovi\'c et al.~\cite{DEW17} provided an improved upper bound of $(2g+3)(2d+1)$ for this subclass of $g$-map graphs. 
\end{remark}

We now consider contact string graphs. A \textit{simple curve} (or \textit{string}) in $\mathbb{R}^2$ is a subset of $\mathbb{R}^2$ homeomorphic to $[0, 1]$. 
A finite set $\mathcal{R}$ of simple curves in $\mathbb{R}^2$ is a \textit{curve contact representation} of a graph $G$ if the interiors of the curves are pairwise disjoint and $G$ is the intersection graph of $\mathcal{R}$. A curve contact representation $\mathcal{R}$ is a \textit{segment contact representation} if each curve of $\mathcal{R}$ is a line segment.
A graph $G$ is a \textit{contact string graph} (\textit{contact segment graph}) if
it admits a curve contact representation (segment contact representation). For any curve $P$ of a curve contact representation $\mathcal{R}$, we denote by $\partial(P)$ the set consisting of its two endpoints and by $\mathring{P}$ its \textit{interior}, i.e., the set $P\setminus\partial(P)$. A point in $\mathbb{R}^2$ is a \textit{contact point} of a contact representation $\mathcal{R}$ if it is contained in at least two curves of $\mathcal{R}$. Observe that a curve contact representation of a contact string graph on $n$ vertices has at most $2n$ contact points.

Given that $0$-map graphs coincide with contact graphs of disk homeomorphs in $\mathbb{R}^2$ and have bounded layered tree-independence number (\Cref{maplayered}), the following question is natural: 
Do contact string graphs have bounded layered tree-independence number?
We answer this question in the negative by showing that contact segment graphs have unbounded local tree-independence number, and, hence, unbounded layered tree-independence number.
Recall that, given a width parameter $p$, a graph class $\mathcal{G}$ has \textit{bounded local $p$} if there is a function $f\colon \mathbb{N} \rightarrow \mathbb{N}$ such that for every integer $r \in \mathbb{N}$, graph $G \in \mathcal{G}$, and vertex $v \in V(G)$, the subgraph $G[N^{r}[v]]$ has value of $p$ at most $f(r)$. In \cite{GMY24}, it is shown that if every graph in a class $\mathcal{G}$ has layered tree-independence number at most $\ell$, then $\mathcal{G}$ has bounded local tree-independence number with $f(r) = \ell(2r + 1)$.

\begin{figure}[h!]
\centering
\includegraphics[keepaspectratio,width=0.5 \textwidth]{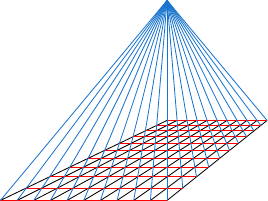}
\quad\quad\quad\quad\includegraphics[keepaspectratio,width=0.38\textwidth]{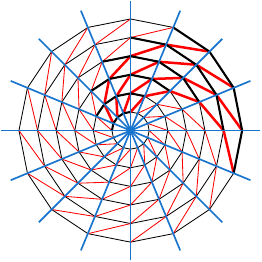}
\caption{Examples of line segments in $\mathbb{R}^2$ yielding contact segment graphs with radius $2$ and unbounded tree-independence number.
Examples on the right extend the ones on the left, in the sense that sufficiently large examples of the right type contain any fixed example of the left type (as indicated by the bold part of the figure on the right).}
\label{contactstringlocal}
\end{figure}

\begin{proposition}\label{contactlocal}
The class of contact segment graphs has unbounded local tree-independence number.
\end{proposition}

\begin{proof}
Consider the family of line segments in $\mathbb{R}^2$ depicted in the left part of \Cref{contactstringlocal}. 
Each segment is either blue, red, or black. 
The blue segments are the only segments whose interiors may contain endpoints of other segments (in particular, the red and black segments are significantly shorter than the blue ones).
Let $G$ be the corresponding contact segment graph.
Since the blue segments all share a point, they correspond to a clique $C$ in $G$.
Moreover, since every red or black segment shares a point with a blue segment, every vertex in $C$ is at distance at most $2$ from every other vertex.
Hence, $G$ has radius $2$.

Recall now that the line graph of a graph $H$ is the graph with vertex set $E(H)$, in which two distinct vertices are adjacent if and only if the corresponding edges of $H$ share an endpoint. Observe that the subgraph $G-C$ induced by vertices corresponding to the red and black segments is isomorphic to the line graph of an $n\times n$ grid (the graph in the figure has $n = 11$). It is well known that grids have unbounded treewidth and so line graphs of grids have unbounded tree-independence number (see, e.g., \cite[Theorem~3]{BMPY25}). Using an obvious generalization of the family of line segments in $\mathbb{R}^2$ depicted in the left part of \Cref{contactstringlocal}, we obtain a family of contact segment graphs with radius $2$ and with unbounded tree-independence number, hence with unbounded local tree-independence number.
\end{proof}

\Cref{contactlocal} and \Cref{maplayered} imply that there exist contact string graphs which are not $0$-map graphs. Our next result implies that these two graph classes are in fact incomparable. 

\begin{figure}[h!]
\centering\includegraphics[]{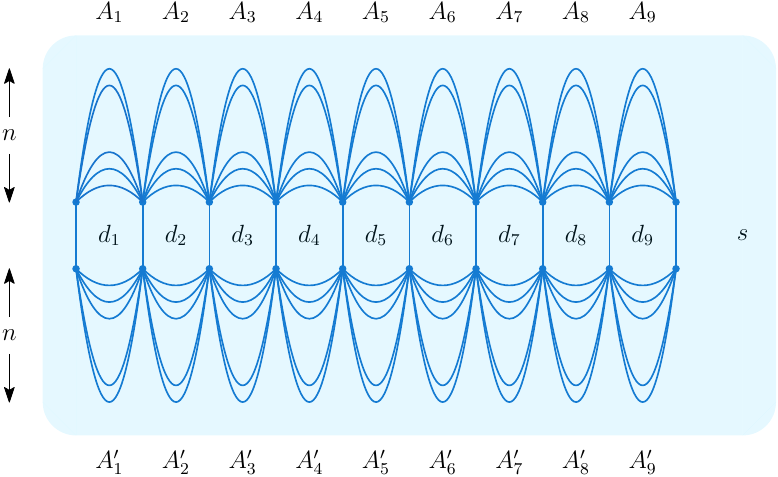}
\caption{A representation of a $0$-map graph $G$ which is not contact string. 
Each $A_i$ and $A'_i$ consists of $n$ nations, where $n$ is sufficiently large (see \Cref{mapnotcontact}). 
The graph $G$ has $18n + 9 + 1$ vertices, corresponding to the faces of the plane graph $G_0$ drawn in blue, which are all nations.}
\label{mapgraph}
\end{figure}

\begin{proposition}\label{mapnotcontact} There exist $0$-map graphs which are not contact string graphs.
\end{proposition}

\begin{proof}
We show that the $0$-map graph in \Cref{mapgraph} is not a contact string graph.
We will make use of the fact that, in every curve contact representation of a large clique, many strings have a common endpoint. 
In order to formalize this, we introduce some notation from \cite{DGMR18}. 
Let $\mathcal{R}$ be a curve contact representation of a contact string graph $G=(V,E)$ where, for each $u \in V(G)$, the curve $P_u \in \mathcal{R}$ corresponds to the vertex $u$ and has endpoints $q_u^1$ and $q_u^2$. 
For each $u \in V(G)$, let $w_u = w_u^1 + w_u^2$, where for $i \in \{1,2\}$, \[w_u^i = |\{P \in \mathcal{R}\colon q_u^i \in \mathring{P}\}| + \frac{1}{2}|\{P \in \mathcal{R}\colon P \neq P_u \ \mbox{and} \ q_u^i \in \partial(P)\}|.\] 
Note that, since the interiors of the curves are pairwise disjoint, $|\{P \in \mathcal{R}\colon q_u^i \in \mathring{P}\}|\le 1$ for each $u \in V(G)$ and $i\in \{1,2\}$.
Since for every edge $uv \in E$ either an endpoint of $P_u$ belongs to the interior of $P_v$, or $P_u$ and $P_v$ have a common endpoint, it is easy to see that $|E| \leq \sum_{u \in V}w_u$. 

\begin{nitem}\label{contactpoint}
    Let $c < 1/2$ and let $n > 3/(1-2c)$.
    Then, for every curve contact representation $\mathcal{R}$ of $K_n$, there exists a point $p \in \mathbb{R}^2$ that is an endpoint of more than $cn$ strings of $\mathcal{R}$.
\end{nitem}

\begin{claimproof}[Proof of \eqref{contactpoint}] 
    Suppose, to the contrary, that every point $p \in \mathbb{R}^2$ is an endpoint of at most $cn$ strings of $\mathcal{R}$. 
    This implies that, for each $u \in V(K_n)$ and $i\in \{1,2\}$, \[w_u^i \leq 1 + \frac{1}{2}(cn-1) = \frac{cn}{2} + \frac{1}{2}\] and so \[|E(K_n)| \leq \sum_{u \in V(K_n)}w_u \leq 2n\bigg(\frac{cn}{2} + \frac{1}{2}\bigg) = n(cn + 1) < \frac{n(n-1)}{2},\] a contradiction. 
\end{claimproof}

Consider the $0$-map graph $G$ in \Cref{mapgraph}, where $n$ is a sufficiently large power of $24$, say $n=24^{10}$, and all the faces are nations. 
The $0$-map graph $G$ has $18n + 9 + 1$ vertices. 
Suppose, to the contrary, that $G$ has a curve contact representation. 
By taking $c = 1/3$ in \eqref{contactpoint}, there exist a point $p$ that is an endpoint of (more than) $n/3$ strings of $A_1$ and a point $p_2$ that is an endpoint of (more than) $n/3$ strings of $A_2$. 
Let $A_1(p_1) \subseteq A_1$ and $A_2(p_2) \subseteq A_2$ be such sets of strings of size $n/3$. Note that, since the sets $A_1$ and $A_2$ are disjoint, so are $A_1(p_1)$ and $A_2(p_2)$. 

We claim that there exists a point that is an endpoint of $n/24$ strings of $A_1(p_1)$ and $n/24$ strings of $A_2(p_2)$. 
If $p_1 = p_2$, the claim is immediate. 
Hence, suppose instead that $p_1 \neq p_2$ and that none of $p_1$ and $p_2$ satisfies the claimed property. 
Consider the curve contact representation obtained by restricting the strings to $A_1(p_1) \cup A_2(p_2)$ and let $G'$ be the corresponding complete graph on $2n/3$ vertices.
We argue that there must be a point $q_1$, distinct from $p_1$ and $p_2$, that is an endpoint of at least $n/6$ strings of $A_1(p_1) \cup A_2(p_2)$. 
Indeed, if this is not the case, then for each $u \in A_1(p_1) \cup A_2(p_2)$, 
\[w_u^1 \leq 1 + \frac{1}{2}\bigg(\bigg(\frac{n}{3} -1\bigg)+ \bigg(\frac{n}{24}-1\bigg)\bigg) \ \ \ \mbox{and} \ \ \  w_u^2\le 1 + \frac{1}{2}\bigg(\frac{n}{6}-2\bigg),\] which implies \[|E(G')| \leq \sum_{u\in A_1(p_1) \cup A_2(p_2)}w_u \leq \frac{2n}{3}\cdot\frac{1}{2}\bigg(\frac{n}{3}+\frac{n}{24}+\frac{n}{6}\bigg) < \frac{1}{2}\cdot\frac{2n}{3}\bigg(\frac{2n}{3}-1\bigg),\] a contradiction.   
Therefore, let $q_1$ be a point distinct from $p_1$ and $p_2$ that is an endpoint of at least $n/6$ strings of $A_1(p_1) \cup A_2(p_2)$. 
If $q_1$ is an endpoint of $n/12$ strings of $A_1(p_1)$ and $n/12$ strings of $A_2(p_2)$, we are done. 
Otherwise, without loss of generality, $q_1$ is an endpoint of more than $n/12$ strings of $A_1(p_1)$. 
But since $n/12 > 3$ and $A_1(p_1) \cup A_2(p_2)$ is a clique, each of the $n/3$ strings of $A_2(p_2)$ must contain one of $p_1$ and $q_1$. 
Moreover, at most two of the strings of $A_2(p_2)$ contain one of $p_1$ and $q_1$ as an interior point, and so at least $n/3 - 2 > n/12$ strings of $A_2(p_2)$ have one of $p_1$ and $q_1$ as an endpoint.
By the pigeonhole principle, this implies that one of $p_1$ and $q_1$ is an endpoint of $n/24$ strings of $A_1(p_1)$ and $n/24$ strings of $A_2(p_2)$, as claimed. 

By the paragraph above, let $b_1$ be a point that is an endpoint of $n/24$ strings of $A_1(p_1)$ and $n/24$ strings of $A_2(p_2)$. 
Consider now these $n/24$ strings of $A_2(p_2)$ and any set of $n/24$ strings of $A_3$ with a common endpoint (whose existence is guaranteed by \eqref{contactpoint}). 
Applying the argument from the previous paragraph, we obtain another point $b_2$ which is the endpoint of $n/(24^2)$ strings of $A_2(p_2)$ and $n/(24^2)$ strings of $A_3$. 
Note that $b_2$ is distinct from $b_1$, as $A_1$ is anticomplete to $A_3$. 
Recursively, we obtain a sequence of distinct points $b_1,\ldots,b_9$ such that, for each $1 \leq i \leq 8$, at least $n/(24^i)$ strings of $A_{i+1}$ have $b_i$ as an endpoint and at least $n/(24^{i+1})$ of these strings have $b_{i+1}$ as an endpoint. 
Hence, for each $1\leq i\leq 8$, there exists a subset $B_{i+1}$ of at least $n/(24^{9})$ strings of $A_{i+1}$ whose endpoints are $b_i$ and $b_{i+1}$. We proceed similarly for strings in $A'_1, \ldots, A'_9$. 
For each $1\leq i \leq 8$, we find a subset $C_{i+1}$ of at least $n/(24^{9})$ strings of $A'_{i+1}$ whose endpoints are $c_i$ and $c_{i+1}$, for some distinct points $c_1,\dots,c_9$. 
Note that $c_i \neq b_j$, for any pair $i, j\in \{1,\ldots,9\}$.

For any point $q$ and string $\ell$, define now the function $f(q,\ell)$ as follows: $f(q,\ell) = 1$ if $\ell$ contains $q$ (either as endpoint or interior point), and $f(q,\ell) = 0$ otherwise. Consider the string $s$ corresponding to the outer nation and observe that $s$ must contain a point of each string of $B_i$, for every $2 \leq i \leq 9$. It follows that $f(b_i,s)+ f(b_{i+1},s) \geq 1$, for every $1 \leq i \leq 8$, from which $\sum_{i=2}^9 f(b_i,s) \geq 4$. Similarly, $\sum_{i=2}^9 f(c_i,s) \geq 4$. 

Consider now the string $d_i$, for any $3\leq i \leq 8$. It must contain a point of each string of $B_{i-1}$ and $B_{i+1}$. But the endpoints of each string of $B_{i-1}$ are $b_{i-2}$ and $b_{i-1}$, and the endpoints of each string of $B_{i+1}$ are $b_{i}$ and $b_{i+1}$. It follows that $f(b_{i-2},d_i)+ f(b_{i-1},d_i)\geq 1$ and $f(b_{i},d_i)+ f(b_{i+1},d_i)\geq 1$, from which $\sum_{j=2}^9 f(b_j,d_i) \geq 2$.  Therefore, $\sum_{i=3}^8 \sum_{j=2}^9f(b_j,d_i) \geq 6\cdot 2 =12$. Similarly, $\sum_{i=3}^8\sum_{j=2}^9 f(c_j,d_i) \geq 6\cdot 2 =12$. 

Combining the bounds from the previous two paragraphs, we obtain that
\begin{align*}
& \sum_{i=2}^9 \sum_{\ell\in \{s,d_3,\ldots,d_8\}}\big(f(b_i,\ell)+f(c_i,\ell)\big) \\ 
=& \sum_{i=2}^9 f(b_i,s)+ \sum_{i=2}^9 f(c_i,s) + 
   \sum_{j=2}^9\sum_{i=3}^8 f(b_j,d_i) + \sum_{j=2}^9\sum_{i=3}^8 f(c_j,d_i) \\
\geq& \ 4+4+12+12 = 32. 
\end{align*}
Now, the strings in $\{s,d_3,\ldots,d_8\}$ give rise to at most $14$ distinct endpoints. Hence, among all instances with $f(b_i,\ell) = 1$ or $f(c_i,\ell) = 1$, at most $14$ are such that one of $b_i$ and $c_i$ is an endpoint of $\ell$, and so there are at least $32-14=18$ instances for which one of $b_i$ and $c_i$ is an interior point of $\ell$. But then, by the pigeonhole principle, one of the $16$ distinct points $b_2,\ldots,b_9, c_2,\ldots,c_9$ must be an interior point of at least two strings, a contradiction.
\end{proof}

We conclude this section by observing that a proper subclass of contact string graphs known as one-sided contact string graphs \cite{Hli95,Hli98,H98} is in fact contained in the class of $0$-map graphs. 

A curve contact representation $\mathcal{R}$ is \textit{one-sided} if each of its contact points $o$ is one-sided, i.e., either all curves containing $o$ end in $o$, or $o$ lies in the interior of a curve and all other curves containing $o$ lie on the same side of that curve. This property of contact representations is transferred to contact graphs, and we refer to contact graphs as one-sided in the obvious sense. As observed by  Hlinen{\'{y}}~\cite{Hli95}, not every contact string graph has a one-sided curve contact representation.  

\begin{lemma}\label{contactsegmenttoplanar} Every one-sided contact string graph is a $0$-map graph.  In particular, every one-sided contact string graph has layered tree-independence number at most $9$. 
\end{lemma}

\begin{proof} Let $G$ be a one-sided contact string graph with curve contact representation $\mathcal{R}$. It is enough to show that there exists a planar bipartite graph $H$ with bipartition $(X, Y)$ such that $G$ is the half-square graph $H^2[X]$. We build such a graph $H$ as follows. The set $X$ corresponds to the curves in $\mathcal{R}$, the set $Y$ corresponds to the contact points of $\mathcal{R}$, and there is an edge between $x \in X$ and $y \in Y$ if and only if the curve corresponding to $x$ contains the point corresponding to $y$ (this is known as the \textit{incidence graph} of $\mathcal{R}$ \cite{dFOdM07,Hli98}). It is well known that $H$ is planar and it is easy to see that $G$ is indeed the half-square graph $H^2[X]$ (see \Cref{incidencegraph}).
Hence, $G$ is a $0$-map graph and the claimed bound on the layered tree-independence number follows from \Cref{maplayered}.
\end{proof}

\begin{figure}[h]
\centering
\includegraphics[scale=0.8]{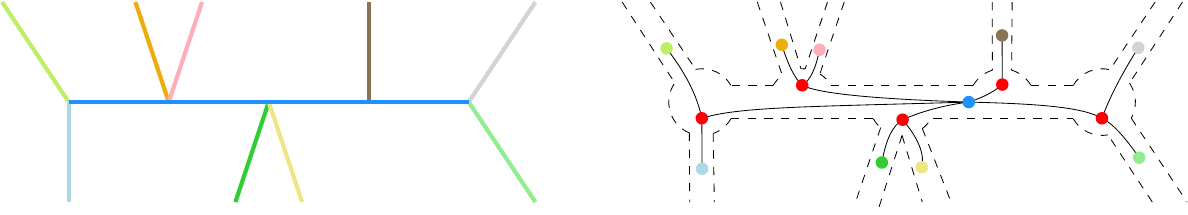}
\caption{Construction of a planar bipartite witness $H$ of a one-sided contact string graph $G$ in the proof of \Cref{contactsegmenttoplanar}. Vertices in red correspond to contact points.}
\label{incidencegraph}
\end{figure}

\subsection{Hyperbolic uniform disk graphs}\label{sec:hyperbolic}

We recall some useful facts about the hyperbolic plane and refer the reader to \cite{RR} for a thorough discussion. We denote the hyperbolic distance by $d_{\mathbb{H}^2}\colon \mathbb{H}^2 \times \mathbb{H}^2 \rightarrow \mathbb{R}$. We consider the \textit{polar-coordinate model} of $\mathbb{H}^2$ defined as follows. We fix a designated point $o \in \mathbb{H}^2$, called \textit{pole}, together with a reference ray starting at $o$, called \textit{polar axis}. Any point $p \in \mathbb{H}^2$ is identified by the pair $(b(p),\theta(p))$, where $b(p)$ denotes the hyperbolic distance between $p$ and $o$ (i.e., $b(p) = d_{\mathbb{H}^2}(p,o)$), and $\theta(p)$ denotes the angle from the polar axis to the ray from $o$ through $p$, where $\theta(p) \in [0, 2\pi)$. 
The pair $(b,\theta)$ refers to the point $p \in \mathbb{H}^2$ with $b(p) = b$ and $\theta(p) = \theta$. 

Since it is $\exists\mathbb{R}$-complete to decide if a given graph belongs to HUDG \cite{BBDJ23}, in the following results we assume that every hyperbolic uniform disk graph is given together with a geometric realization consisting of the polar coordinates of the disk centers. Moreover, we assume that the following operation can be done in constant time: For a disk $O$, compute $\inf_{x\in O}\theta(x)$ and $\sup_{x\in O}\theta(x)$.

\begin{theorem}\label{layeredhyperbolic}
Let $G$ be a hyperbolic uniform disk graph with radius $r$ on $n$ vertices. It is possible to compute, in $\mathcal{O}(n\log n)$ time, a tree decomposition $\mathcal{T}$ with $|\mathcal{T}| \leq 2n$ and a layering $(V_1, V_2, \ldots)$ of $G$ witnessing layered tree-independence number $6\lceil\frac{r}{\tanh r} \rceil$. 
In particular, $G$ has layered tree-independence number at most $6\lceil\frac{r}{\tanh r} \rceil$. 
\end{theorem}

\begin{proof} For each $v \in V(G)$, let $O_v$ denote the disk corresponding to the vertex $v$. We fix a polar-coordinate system for $\mathbb{H}^2$ in such a way that no disk intersects the polar axis, and so the polar angle of any point in a disk belongs to the open interval $(0, 2\pi)$. To see that such choice of polar-coordinate system is possible, start from any pole $o'$ and polar axis $\overrightarrow{k}$. If some disks intersect $\overrightarrow{k}$, consider a point $p$ of intersection with maximum distance from $o'$. Then simply choose the new pole $o$ as a point on $\overrightarrow{k}$ with $d_{\mathbb{H}^2}(o,o') > d_{\mathbb{H}^2}(p,o')$, and as a polar axis the subray of $\overrightarrow{k}$ starting at $o$. 
    
\begin{sloppypar}
For each $i \in \mathbb{N}$, let $R_i = \{x\in \mathbb{H}^2\colon (2i-2)r < b(x) \leq 2ir\}$ and let $V_i = \{v\in V(G)\colon  \mbox{the center of } O_v \mbox{ belongs to } R_i\}$. Clearly, $(V_1, V_2, \ldots)$ is a partition of $V(G)$. 
Furthermore, if two vertices $u$ and $v$ of $G$ are adjacent, then the centers of the corresponding disks are at distance at most $2r$ and so the triangle inequality implies that these centers either belong to the same $R_i$ or to two distinct consecutive $R_i$'s. Hence, $(V_1, V_2, \ldots)$ is a layering of $G$.
\end{sloppypar}

We now build a tree decomposition of $G$. For each $v \in V(G)$, let $\theta^-(v) = \inf_{x\in O_v}\theta(x)$ and let $\theta^+(v) = \sup_{x\in O_v}\theta(x)$. 
Since no disk intersects the polar axis and since the disks are closed, the polar angles of points in $O_v$ spread from $\theta^-(v)$ to $\theta^+(v)$, that is, for any $\theta_0$ with $\theta^-(v) \leq \theta_0 \leq \theta^+(v)$, there exists $x\in O_v$ such that $\theta(x) = \theta_0$. Let now $z_1,\ldots,z_{m}$ be an increasing ordering of $\{\theta^-(v)\colon v\in V(G)\} \cup \{\theta^+(v)\colon v\in V(G)\}$, eliminating any repetition. Let $S_i = \{x\in \mathbb{H}^2\colon \theta(x) = z_i\}$ be the ray with polar angle $z_i$. Let $T$ be a path with $m$ vertices $t_1,\ldots,t_{m}$ and associate each node $t_i$ with the bag $X_{t_i} = \{v \in V(G)\colon O_v \cap S_i \neq \varnothing\}$. 

\begin{nitem}\label{treedechyper}
    $\mathcal{T} = (T,\{X_{t_i}\}_{1\leq i \leq m})$ is a tree decomposition of $G$.
\end{nitem}

\begin{claimproof}[Proof of \eqref{treedechyper}] Consider first (T1). Let $v\in V(G)$ be an arbitrary vertex. We have that $\theta^-(v) = z_i$ for some $i$, and so $v\in X_{t_i}$.
        
Consider now (T2). Let $v_1$ and $v_2$ be two adjacent vertices and fix a point $x \in O_{v_1} \cap O_{v_2}$. Clearly, there exists an index $i$ such that $z_i \leq \theta(x) \leq z_{i+1}$. If $\theta(x) = z_{i+1}$, then both $O_{v_1}$ and $O_{v_2}$ intersect $S_{i+1}$ and hence $v_1$ and $v_2$ both belong to $X_{t_{i+1}}$. If, however, $\theta(x) \neq z_{i+1}$, we claim that both $O_{v_1}$ and $O_{v_2}$ intersect $S_i$. Suppose, to the contrary, that $O_{v_j}$ does not intersect $S_i$, for some $j \in \{1, 2\}$. Then, it must be that $z_i < \theta^-(v_j) \leq \theta(x) < z_{i+1}$, contradicting the definition of $z_1,\ldots,z_{m}$. 

Consider finally (T3). Let $v\in V(G)$ be an arbitrary vertex. There exist indices $i < j$ such that $\theta^-(v)=z_i$ and $\theta^+(v)=z_j$. We have that $v \in X_{t_k}$ if and only if $O_v$ intersects $S_k$, hence, if and only if $i \leq k \leq j$. Therefore, $v$ belongs to bags associated with consecutive nodes of $T$.
\end{claimproof}

\begin{nitem}\label{intersechyper}
    For each bag $X_{t_i}$ and layer $V_j$, $\alpha(G[X_{t_i} \cap V_j]) \leq 6\lceil\frac{r}{\tanh r} \rceil$.
\end{nitem}

\begin{claimproof}[Proof of \eqref{intersechyper}] Fix an arbitrary bag $X_{t_i}$ and layer $V_j$ and let $v \in X_{t_i} \cap V_j$ be an arbitrary vertex. Let $c$ be the center of $O_v$. By definition, $O_v$ intersects $S_i$ and the center $c$ of $O_v$ belongs to $R_j$. Extend $S_i$ to a straight line $\ell$ and let $s$ be the intersection between $\ell$ and the line perpendicular to $\ell$ through $c$. Note that if $c$ lies on $\ell$, then $s = c$. Since $s$ is the point on $\ell$ minimizing the distance to $c$ (see, e.g., \cite[Theorem~2.2]{RR}) and since $O_v$ intersects $S_i$, it must be that $d_{\mathbb{H}^2}(c,s) \leq r$ and so $O_v$ contains $s$. 

We now show that $(2j-3)r \leq b(s) < 2jr$. Suppose first that $b(s) < (2j-3)r$. The triangle inequality implies that $b(c) = d_{\mathbb{H}^2}(c,o) \leq d_{\mathbb{H}^2}(c,s) + d_{\mathbb{H}^2}(s,o) < r + (2j-3)r = (2j-2)r$, from which $c \notin R_j$, a contradiction. Suppose finally that $b(s) \geq 2jr$ and consider again the right triangle with vertices $s$, $c$, and $o$, and whose hypotenuse is $co$. Since the hypotenuse is longer than either of the other two sides (see, e.g., \cite[Corollary~1, p.~34]{RR}), we have that $d_{\mathbb{H}^2}(c,o) > d_{\mathbb{H}^2}(s,o) \geq 2jr$, from which $c \notin R_j$, a contradiction. 

Let now $h = 3\lceil\frac{r}{\tanh r} \rceil$. For each $0\leq k \leq h$, let $\ell_k$ be the line perpendicular to $\ell$ at the point $s_k = (2jr-k\tanh r, z_i) \in \ell$. Let $C$ and $D$ be the two equidistant curves containing all points at distance exactly $r$ from the straight line $\ell$ on either side of $\ell$, respectively. Each $\ell_k$ intersects $C$ in a unique point $c_k$ and intersects $D$ in a unique point $d_k$. Consider now the region $R$ of $\mathbb{H}^2$ bounded by $C$ and $D$ and containing the straight line $\ell$. The center $c$ of $O_v$ must lie in $R$, as $d_{\mathbb{H}^2}(c,s) \leq r$. In fact, $c$ belongs to the subregion $R'$ of $R$ bounded by the line segments $c_0d_0,c_hd_h$, the portion of $C$ with endpoints $c_0$ and $c_{h}$, and the portion of $D$ with endpoints $d_0$ and $d_h$. Indeed, if $c \notin R'$, then the line through $c$ perpendicular to $\ell$, which intersects $\ell$ in the point $s$ with $2jr-h\tanh r \leq (2j - 3)r \leq b(s) < 2jr$, must intersect one of $\ell_0$ and $\ell_h$, contradicting the fact that no two parallel lines intersect. 

We now further split the region $R'$ as follows. For each $0 \leq k \leq h-1$, consider the region bounded by the line segments $c_ks_k,s_ks_{k+1},s_{k+1}c_{k+1}$ and the portion of $C$ with endpoints $c_k$ and $c_{k+1}$; similarly, consider the regions where $c_k$ and $c_{k+1}$ are replaced by $d_k$ and $d_{k+1}$, respectively. 
The boundary of each of these regions is a Saccheri quadrilateral,\footnote{A \textit{Saccheri quadrilateral} is a quadrilateral $abcd$ with right angles at $b$ and $c$ and legs $ab$, $dc$ of equal length (see, e.g., \cite{RR}).} with base of length $\tanh r$ and legs of length $r$. By the proof of \cite[Lemma~8]{BvdH24}, each such region has diameter at most $2r$ and so there are no two non-intersecting disks of radius $r$ whose centers belong to the same region. Since there are $2h$ such regions, there are at most $2h = 6\lceil\frac{r}{\tanh r} \rceil$ pairwise non-intersecting disks whose corresponding vertices belong to $X_{t_i} \cap V_j$.
\end{claimproof}

This concludes the proof of \Cref{layeredhyperbolic}.
\end{proof}

\Cref{layeredhyperbolic} and \Cref{treealphabound} imply the following.

\begin{corollary}\label{treealphahudg} Every hyperbolic uniform disk graph with radius $r$ on $n$ vertices has tree-independence number $\mathcal{O}(\sqrt{\frac{r}{\tanh r}}\cdot \sqrt{n})$. A tree decomposition witnessing this can be computed in $\mathsf{poly}(n)$ time.
\end{corollary}

\subsection{Spherical uniform disk graphs}\label{sec:spherical}

Let $\mathbb{S}^2$ denote the \textit{unit sphere} in $\mathbb{R}^3$ centered at the origin, i.e., $\mathbb{S}^2 = \{x \in \mathbb{R}^3\colon \Vert x \Vert = 1\}$. 
The section of $\mathbb{S}^2$ by a plane is a \textit{circle} on $\mathbb{S}^2$ which, in the case the plane passes through the origin, is called a \textit{great circle}. For a point $x \in \mathbb{S}^2$, $x^*$ denotes its \textit{antipodal point}, i.e., the intersection point distinct from $x$ of the (Euclidean) line through $x$ and the origin with $\mathbb{S}^2$. For any two distinct points $a, b \in \mathbb{S}^2$, there is a great circle passing through $a$ and $b$ (if $a$ and $b$ are not antipodal, such a great circle is unique; otherwise, we fix one arbitrarily); the length of the shorter of the two arcs of this great circle is the \textit{spherical distance} between $a$ and $b$, denoted by $d_{\mathbb{S}^2}(a, b)$. Hence, the spherical distance between $a$ and $b$ is nothing but the angle between the vectors $a, b \in \mathbb{S}^2$. The function $d_{\mathbb{S}^2}\colon \mathbb{S}^2\times \mathbb{S}^2 \rightarrow \mathbb{R}$ is a metric on $\mathbb{S}^2$. For $0 < r < \pi$ and $v \in \mathbb{S}^2$, the \textit{spherical disk} (or \textit{spherical cap}) of radius $r$ centered at $v$ is the set $\{u \in \mathbb{S}^2\colon d_{\mathbb{S}^2}(u, v) \leq r\}$.

Recall that every unit disk graph is a spherical uniform disk graph. However, we now provide an infinite family of spherical uniform disk graphs which are not unit disk graphs. For $t\geq 1$, let $K_t$ denote the complete graph on $t$ vertices and, for $t \geq 3$, let $C_t$ denote the cycle on $t$ vertices. The disjoint union $G + H$ of two vertex-disjoint graphs $G$ and $H$ is the graph with vertex set $V(G) \cup V(H)$ and edge set $E(G) \cup E(H)$. The complement $\overline{G}$ of a graph $G$ is the graph with vertex set $V(G)$ in which two distinct vertices are adjacent if and only if they are nonadjacent in $G$.

\begin{proposition}\label{spherical-superclass}
    For every positive integer $k$, the graph $\overline{K_2 + C_{2k+1}}$ is a spherical uniform disk graph but not a unit disk graph.
\end{proposition}

\begin{proof}
    Atminas and Zamaraev~\cite[Theorem~4.1]{AZ18} showed that, for every positive integer $k$, the graph $G = \overline{K_2 + C_{2k+1}}$ is not a unit disk graph.
    It remains to show that $G$ is a spherical uniform disk graph.
    We claim that for any $r$ such that $\max\{\frac{\pi}{4},\frac{\pi (k-1)}{2k+1}\} \leq r < \frac{\pi k}{2k+1}$, the graph $G$ can be obtained as the intersection graph of spherical disks of radius $r$ on $\mathbb{S}^2$.
    To show this, we fix a great circle on $\mathbb{S}^2$, which we call equator. First, we place $2k+1$ disks $D_1, \ldots, D_{2k+1}$ with radius $r$ as above evenly along the equator in such a way that their centers lie on the equator and the centers of any two consecutive disks are at (spherical) distance $2\pi/(2k+1)$.
    Note that, since $\frac{\pi (k-1)}{2k+1} \leq r < \frac{\pi k}{2k+1}$, each disk intersects all but the two disks that are furthest away from it (for instance, $D_1$ intersects all disks other than $D_{k+1}$ and $D_{k+2}$).    
    Therefore, the intersection graph corresponding to the disks on the equator is isomorphic to $\overline{C_{2k+1}}$.
    Now, we add two disks of radius $r$ centered at the respective poles of $\mathbb{S}^2$.
    Note that $r < \pi/2$, and hence such pole disks do not intersect.
    On the other hand, since $r \geq \frac{\pi}{4}$, each pole disk intersects all equator disks.
    We conclude that the corresponding intersection graph is indeed isomorphic to $G$.
\end{proof}

We now introduce a polar-coordinate model of $\mathbb{S}^2$, sometimes known as spherical coordinates. Fix a designated point $o \in \mathbb{S}^2$, called \textit{pole}, its antipodal point $o^*$, and a reference geodesic segment $oo^*$ (i.e., a length-$\pi$ arc $\ell$ on $\mathbb{S}^2$ with endpoints $o$ and $o^*$), called \textit{polar axis}. Any point $p\in \mathbb{S}^2\setminus\{o, o^*\}$ is identified by the pair $(b(p), \theta(p))$, where $b(p)$ denotes the spherical distance between $p$ and $o$, and $\theta(p) \in [0, 2\pi)$ denotes the angle from the polar axis to the geodesic segment $op$. The pair $(b,\theta)$ refers to the point $p \in \mathbb{S}^2$ with $b(p) = b$ and $\theta(p) = \theta$. 

Similarly to the hyperbolic case, in the following results we assume that every spherical uniform disk graph is given together with a geometric realization consisting of the polar coordinates of the disk centers. Moreover, we assume that for a disk $O$, computing $\inf_{x\in O}\theta(x)$ and $\sup_{x\in O}\theta(x)$ can be done in constant time.

\begin{theorem}\label{sphericallayered}
Let $G$ be a spherical uniform disk graph with radius $r$ on $n$ vertices. It is possible to compute, in $\mathcal{O}(n\log n)$ time, a tree decomposition $\mathcal{T}$ with $|\mathcal{T}| \leq 2n$ and a layering $(V_1, V_2, \ldots)$ of $G$ witnessing layered tree-independence number $30$. 
In particular, $G$ has layered tree-independence number at most $30$. 
\end{theorem}

\begin{proof} For each $v \in V(G)$, let $O_v$ denote the disk corresponding to the vertex $v$. Fix a polar-coordinate system for $\mathbb{S}^2$, as described above, with pole $o$ and polar axis $\ell$. Note that the polar angle of any point in a disk not intersecting $\ell$ belongs to $(0, 2\pi)$. 

\begin{sloppypar}
For each $i \in \mathbb{N}$, let $R_i = \{x\in \mathbb{S}^2\colon (2i-2)r < b(x) \leq 2ir\}$ and let $V_i = \{v\in V(G)\colon  \text{the center of} \ O_v \ \text{belongs to} \ R_i\}$. Clearly, $(V_1,V_2,\ldots)$ partition $V(G)$. Furthermore, if two vertices $u$ and $v$ of $G$ are adjacent, then the centers of the corresponding disks are at distance at most $2r$ and so the triangle inequality implies that these centers either belong to the same $R_i$ or to two distinct consecutive $R_i$'s. 
Hence, $(V_1, V_2, \ldots)$ is a layering of $G$.
\end{sloppypar}

We now build a tree decomposition of $G$. Let $U$ be the set of vertices whose corresponding disks intersect $\ell$ and let $U' = V(G) \setminus U$. For each $v \in U'$, let $\theta^-(v) = \inf_{x\in O_v}\theta(x)$ and let $\theta^+(v) = \sup_{x\in O_v}\theta(x)$. 
Observe that, for each $v \in U'$, the polar angles of points in $O_v$ spread from $\theta^-(v)$ to $\theta^+(v)$, that is, for any $\theta_0$ with $\theta^-(v) \leq \theta_0 \leq \theta^+(v)$, there exists $x\in O_v$ such that $\theta(x) = \theta_0$. Let $z_0 = 0$ and, in case $U' \neq \varnothing$, let $z_1,\ldots,z_{m}$ be an increasing ordering of $\{\theta^-(v)\colon v\in U'\} \cup \{\theta^+(v)\colon v\in U'\}$, eliminating any repetition (with $m = 0$ if $U' = \emptyset$).
For any $0\le i\le m$, let $S_i = \{x\in \mathbb{S}^2\colon \theta(x) = z_i\}$. 
Note that, for $i = 0$, $S_i$ is just the polar axis $\ell$. Let $T$ be a path with $m+1$ vertices $t_0, t_1,\ldots,t_{m}$ and associate each node $t_i$ with the bag $X_{t_i} = \{v \in U'\colon O_v \cap S_i \neq \varnothing\} \cup U$. Note that, if $U' = \varnothing$, then $U = V(G)$ and $T$ consists of a single vertex $t_0$ and $X_{t_0} = V(G)$. 

\begin{nitem}\label{treedecsphere}
    $\mathcal{T} = (T,\{X_{t_i}\}_{0\leq i \leq m})$ is a tree decomposition of $G$.
\end{nitem}

\begin{claimproof}[Proof of \eqref{treedecsphere}] Consider first (T1). Let $v\in V(G)$ be an arbitrary vertex. If $v \in U$, then $v$ belongs to every bag of $\mathcal{T}$. If $v \in U'$, then $\theta^-(v) = z_i$ for some $i\in \{1,\ldots, m\}$, and so $v\in X_{t_i}$.

Consider now (T2). Let $v_1$ and $v_2$ be two adjacent vertices. The property trivially holds if at least one of $v_1$ and $v_2$ belongs to $U$, as $U$ is included in every bag of $\mathcal{T}$. Therefore, suppose that none of $v_1$ and $v_2$ belongs to $U$. Fix a point $x \in O_{v_1} \cap O_{v_2}$. Clearly, there exists an index $i$ such that $z_i \leq \theta(x) \leq z_{i+1}$. If $\theta(x) = z_{i+1}$, then both $O_{v_1}$ and $O_{v_2}$ intersect $S_{i+1}$. If, however, $\theta(x) \neq z_{i+1}$, then both $O_{v_1}$ and $O_{v_2}$ intersect $S_i$, since otherwise $O_{v_j}$ does not intersect $S_i$ for some $j \in \{1, 2\}$, and so $z_i < \theta^-(v_j) \leq \theta(x) < z_{i+1}$, contradicting the definition of $z_1,\ldots,z_{m}$. 

Consider finally (T3). Let $v\in V(G)$ be an arbitrary vertex. If $v \in U$, the property trivially holds. If $v \in U'$, there exist indices $i < j$ such that $\theta^-(v)=z_i$ and $\theta^+(v)=z_j$, and so $v$ belongs precisely to the bags $X_{t_k}$ with $i \leq k \leq j$.
\end{claimproof}

\begin{nitem}\label{intersecsphere}
    For each bag $X_{t_i}$ and layer $V_j$, $\alpha(G[X_{t_i} \cap V_j]) \leq 30$.
\end{nitem}

\begin{claimproof}[Proof of \eqref{intersecsphere}] Fix an arbitrary bag $X_{t_i}$ and layer $V_j$. We show that there exists a disk of radius $4r$ containing all disks $O_v$ with $v \in X_{t_i} \cap V_j \cap U'$. To this end, take an arbitrary $v \in X_{t_i} \cap V_j \cap U'$ and let $c$ be the center of $O_v$. Let $s \in O_v \cap S_i$. The triangle inequality implies that $b(s) = d_{\mathbb{S}^2}(s,o) \leq d_{\mathbb{S}^2}(s,c) + d_{\mathbb{S}^2}(c,o) \leq r + 2jr = (2j+1)r$. Similarly, $b(s) = d_{\mathbb{S}^2}(s,o) \geq d_{\mathbb{S}^2}(c,o) - d_{\mathbb{S}^2}(c,s) > (2j-2)r - r = (2j-3)r $. Consider now the arc $\{x \in S_i\colon  (2j-3)r \leq b(s) \leq (2j+1)r\}$ and let $z = ((2j-1)r,z_i)$ be its midpoint. Again by the triangle inequality, we obtain that $d_{\mathbb{S}^2}(z,c) \leq d_{\mathbb{S}^2}(z,s) + d_{\mathbb{S}^2}(s,c) \leq 2r + r = 3r$. This implies that the disk centered at $z$ and with radius $4r$ fully contains the disk $O_v$, as claimed. 

Since $S_0$ stabs all disks $O_v$ with $v \in V_j \cap U$, a similar argument as in the previous paragraph shows that there exists a disk of radius $4r$ containing all disks $O_v$ with $v \in X_{t_i} \cap V_j \cap U$.   

But a disk of radius $r$ has area $2\pi(1-\cos r)$ (see, e.g., \cite[Example~4.1]{MM}) and so the number of pairwise non-intersecting disks $O_v$ with $v \in X_{t_i} \cap V_j \cap U'$ is at most $\frac{2\pi(1-\cos 4r)}{2\pi(1-\cos r)} = \frac{1-\cos 4r}{1-\cos r}$. We now bound this quantity using the fact that $r \in (0, \pi)$. 

\begin{nitem}\label{cosineq}
    $\frac{1-\cos 4r}{1-\cos r} < 16$, for every $r \in (0, \pi)$.
\end{nitem}

To prove \eqref{cosineq}, it is enough to show that the function $f\colon[0,\pi)\to \mathbb{R}$ defined by $f(r) = 15-16\cos r+\cos 4r$ satisfies $f(r) > 0$ for every $r\in (0,\pi)$.
Since $f(0) = 0$, it suffices to show that its derivative $f'$ is positive on $(0,\pi)$. Observe that $f'(r) = 16\sin r -4\sin 4r$, so it suffices to show that $4\sin r > \sin 4r$ for every $r\in (0,\pi)$. But it is known that $|n\sin r| \geq |\sin nr|$ for every $n \in \mathbb{N}$ and $r \in \mathbb{R}$ (this can be shown, for example, by an easy induction). Hence, for every $r\in (0,\pi)$, it holds that $4\sin r \geq \sin 4r$, and it is easy to see that no $r \in (0, \pi)$ achieves equality.     

Applying a similar reasoning as above, the number of pairwise non-intersecting disks $O_v$ with $v \in X_{t_i} \cap V_j \cap U$ is at most $15$. Therefore, there are at most $15+15 = 30$ pairwise non-intersecting disks whose corresponding vertices belong to $X_{t_i} \cap V_j$.   
\end{claimproof}
This concludes the proof of \Cref{sphericallayered}.
\end{proof}

\Cref{sphericallayered} and \Cref{treealphabound} imply the following.

\begin{corollary}\label{treealphasudg} Every spherical uniform disk graph with radius $r$ on $n$ vertices has tree-independence number $\mathcal{O}(\sqrt{n})$. A tree decomposition witnessing this can be computed in $\mathsf{poly}(n)$ time.
\end{corollary}

\section{Bounded clique cover degeneracy}\label{s:ccdeg}

In this section we bound the clique cover degeneracy for the classes discussed in Sections \ref{powersbounds} to \ref{sec:spherical}. The fact that powers of graphs with bounded layered treewidth have bounded clique cover degeneracy will follow from \Cref{cliquecoverdegpower}. 
We now prove the following bound for $0$-map graphs, which extends the tight bound of $3$ for the subclass of planar graphs shown by Ye and Borodin~\cite{YB12}.

\begin{proposition}\label{0-map-theta-degen}
    Every $0$-map graph has a vertex whose neighborhood is the union of $3$ cliques.
\end{proposition}

\begin{proof}
    Let $G$ be a $0$-map graph and let $G_0$ be a plane graph (i.e., a graph of Euler genus $0$) whose faces are either nations or lakes, where nations correspond to the vertices of $G$, two vertices being adjacent in $G$ if and only if the corresponding nations of $G_0$ share at least a vertex.

    Let $X$ be the set of \emph{large corners} of $G_0$, that is, points of the plane that are shared by at least $4$ nations.
    We create a graph $H$ as follows.
    For each nation $c_i$ of $G_0$, we create a \textit{nation vertex} $c_i$ in $H$ and, for each large corner of $G_0$, we create a \textit{corner vertex} $x$ in $H$. The adjacencies in $H$ are defined as follows. For every pair of neighboring (i.e., sharing at least one vertex) nations $c_i$ and $c_j$ of $G_0$ that do not share a large corner, we add an edge $c_ic_j$ in $H$.
    For each large corner of $G_0$, we add an edge in $H$ between the corresponding corner vertex $x$ and each nation vertex $c_i$ corresponding to a nation containing the large corner. No other edges belong to $E(H)$. In particular, the set of corner vertices of $H$ forms an independent set. Moreover, if two nation vertices $v_i$ and $v_j$ of $H$ are both adjacent to a corner vertex $x$, then $v_i$ and $v_j$ are not adjacent in $H$, and so the neighborhood of each corner vertex is an independent set of size at least $4$.

    Observe now that if a vertex $p$ of $G_0$ is shared by exactly three nations, then we can assume that every two of these three nations share a private point of the plane along their borders, by slightly changing the plane graph $G_0$ locally around the point $p$ if necessary. This implies that $H$ is planar. 
    Moreover, it is not difficult to see that $G$ is isomorphic to the graph obtained from the neighborhood cliquification on $(H,X)$ followed by deleting $X$.
        
    We can finally show that there exists a vertex $v$ of $G$ whose neighborhood is the union of $3$ cliques. Recall that every planar graph contains a vertex whose neighborhood is the union of $3$ cliques~\cite{YB12}.
    Since $H$ is planar, there is a vertex $v$ in $H$ whose neighborhood is the union of $3$ cliques, say $C_1$, $C_2$, and $C_3$.
    Since the neighborhood of each corner vertex is an independent set of size at least $4$ in $H$, the vertex $v$ cannot be a corner vertex.
    Hence, $v \in V(H) \setminus X = V(G)$.
    Let $X_i$ be the set of corner vertices in $C_i$ (that is, $X_i = X\cap C_i$).
    Note that, since $X$ is an independent set in $H$, we have $|X_i|\le 1$. 
    For $i\in \{1,2,3\}$, let $C^*_i = (C_i \setminus X_i) \cup (\bigcup_{x \in X_i} N_H(x))$.
    Note that each $C^*_i$ is a subset of $V(G)$, as it does not contain corner vertices (which form an independent set).
    Furthermore, if $X_i$ is not empty, then $X_i = \{x\}$ for some $x\in V(H)$, and every vertex in $C^*_i$ is adjacent to $x$ in $H$.
    Therefore, since $G$ is isomorphic to the graph obtained from the neighborhood cliquification on $(H,X)$ followed by deleting $X$, each set $C^*_i$ is a clique in $G$, and $N_G(v) = C^*_1 \cup C^*_2 \cup C^*_3$.
    Therefore, $v$ is a vertex of $G$ whose neighborhood is the union of $3$ cliques.
\end{proof}

The clique cover degeneracy of hyperbolic uniform disk graphs is at most $3$, as shown by Bl{\"{a}}sius et al.~\cite{BvdH24}, and we now bound that of spherical uniform disk graphs.  

\begin{proposition}\label{sphericalcliquecover}
    Every spherical disk graph has a vertex whose neighborhood is the union of $6$~cliques.
\end{proposition}
\begin{proof}

Let $G$ be a spherical disk graph and let $\mathcal{D}$ be the corresponding collection of spherical disks on $\mathbb{S}^2$. Fix a point $p \in \mathbb{S}^2$ not on the boundary of any spherical disk of $\mathcal{D}$ and apply a stereographic projection $\varphi$ from $p$. It is well known that circles on $\mathbb{S}^2$ not containing $p$ are sent to circles on $\mathbb{R}^2$ (see, e.g., \cite[Theorem~1.11]{MM}). Consequently, disks on $\mathbb{S}^2$ containing $p$ in their interior are sent to complements of open disks on $\mathbb{R}^2$. Consider now the subcollection of spherical disks in $\mathcal{D}$ not containing $p$; we may assume it is non-empty, for otherwise $G$ is a complete graph and the statement is trivial. Fix one of such spherical disks $A$ such that its image $\varphi(A)$ is a Euclidean disk with the smallest radius. We aim to show that the neighborhood of $A$ can be covered by six cliques. 

Let $B \in \mathcal{D}$ be a spherical disk containing $p$ and intersecting $A$ in a point $s \neq p$. Clearly, $\varphi(B)$ intersects $\varphi(A)$ in $\varphi(s)$. We draw a Euclidean disk $C$, of the same radius as $\varphi(A)$, inside the region $\varphi(B)$ and containing the point $\varphi(s)$, and then replace $\varphi(B)$ with $C$. Applying this to every spherical disk containing $p$ and intersecting $A$, we obtain a collection of objects such that every object intersecting $\varphi(A)$ is a Euclidean disk of radius at least that of $\varphi(A)$. Let $G'$ be the intersection graph of such collection. Since the neighborhood of every smallest disk in a Euclidean disk graph is the union of $6$ cliques (see, e.g., \cite[Lemma~6]{KT14}), the neighborhood of $\varphi(A)$ can be covered by six cliques in $G'$. 
Since $G'[N_{G'}(\varphi(A))]$ is a spanning subgraph of $G[N_{G}(A)]$, we obtain that $G[N_{G}(A)]$ can be covered by six cliques.
\end{proof}

Even though contact string graphs have unbounded layered tree-independence number, we finally argue that they have clique cover degeneracy at most $4$.

\begin{proposition}\label{csg-ind-deg-4}
    Every contact string graph has a vertex whose neighborhood is the union of $4$ cliques.
\end{proposition}
\begin{proof}
    Let $G$ be a contact string graph. Consider the graph $G'$ obtained from $G$ by removing all the edges that correspond to pairs of strings touching at their endpoints.
    Therefore, for any edge $uv$ of $G'$, the vertices $u$ and $v$ touch at an interior point of either $u$ or $v$.
    We orient the edges of $G'$ as follows: if $u$ and $v$ touch at an interior point of $v$, we orient the edge $uv$ toward $v$; otherwise, we orient $uv$ toward $u$.\footnote{Note that it is possible that both conditions are true at the same time. In this case, the orientation depends on which condition is checked first.}
    Note that every vertex has at most two out-neighbors in $G'$, since each of the endpoints of the corresponding string can touch at most one interior point of another string.
    Thus, there must exist a vertex $x$ with in-degree at most $2$ in $G'$.
    Consider such a vertex $x$ in $G$.
    Let $A$ be the set of neighbors of $x$ such that the contact point with $x$ is an endpoint of the string corresponding to $x$.
    Note that $G[A]$ can be covered with at most two cliques.
    Therefore, since $|N(x) \setminus A| \leq 2$, we conclude that the neighborhood of $x$ can be covered with at most four cliques.
\end{proof}

Given a contact string graph $G$, iteratively removing vertices whose neighborhoods can be covered with $4$ cliques and then greedily coloring the vertices in the opposite order, we obtain a coloring of $G$ with at most $4\omega(G)$ colors.
Hence, the following result holds. It complements the inequality $\chi(G)\le 2\omega(G)$, valid for one-sided contact string graphs \cite{H98}.

\begin{corollary}
If $G$ is a contact string graph, then $\chi(G)\le 4\omega(G)$. 
\end{corollary}

\section{Bounded layered tree-independence number and sublinear-weight clique-based separators}\label{s:towardsquestion}

In this section we collect some evidence toward a positive answer to \Cref{questionltreealpha}. First, we observe that \Cref{questionltreealpha} has a positive answer for every class of bounded layered tree-independence number admitting a subquadratic $\theta$-binding function. Then, we argue that the latter property is easily implied by boundedness of clique cover degeneracy.

\begin{proposition}\label{cliquebased_layeredtreealpha} Let $\mathcal{G}$ be a class of graphs with layered tree-independence number at most $k$ and $\theta$-binding function $f$. Every $n$-vertex graph $G\in \mathcal{G}$ admits clique-based separators of size at most $f(2\sqrt{kn})$.
\end{proposition}

\begin{proof} Let $G \in \mathcal{G}$ be an arbitrary $n$-vertex graph. Since $G$ has layered tree-independence number at most $k$, it admits a tree decomposition $\mathcal{T} = (T, \{X_t\}_{t \in V(T)})$ with independence number at most $2\sqrt{kn}$, thanks to \Cref{treealphabound}. But there exists a balanced separation $(A, B)$ in $G$ such that $A \cap B = X_t$, for some $t \in V(T)$ (see, e.g., \cite[Lemma~7.20]{CFK}), and so $X_t$ is a balanced separator of $G$ with $\alpha(G[X_t]) \leq 2\sqrt{kn}$. Since $\mathcal{G}$ has $\theta$-binding function $f$, we obtain that the balanced separator $X_t$ can be covered with $\theta(G[X_t]) \leq f(2\sqrt{kn})$ cliques, as claimed.   
\end{proof}

\begin{lemma}[Folklore]\label{thetab_cliquecoverdeg}
Let $k \in \mathbb{N}$ and let $\mathcal{G}$ be a class of graphs with clique cover degeneracy at most $k$. The class $\mathcal{G}$ is $\theta$-bounded with $\theta$-binding function $f(x) = kx$.
\end{lemma}

\begin{proof} Fix an arbitrary $G \in \mathcal{G}$. We aim to show that, for every induced subgraph $H$ of $G$, it holds that $\theta(H) \leq k\alpha(H)$. Suppose, to the contrary, that $H$ is a counterexample with the minimum number of vertices. There exists a vertex $v \in V(H)$ such that $N_H(v)$ can be covered with at most $k$ cliques. But then $\theta(H) \leq \theta(H - N_H[v]) + k \leq k\alpha(H - N_H[v]) + k \leq k\alpha(H)$, a contradiction.   
\end{proof}

Applying \Cref{cliquebased_layeredtreealpha} and \Cref{thetab_cliquecoverdeg} gives the following.

\begin{corollary}\label{cliquebasedcor} The following hold:
\begin{itemize}
    \item[(a)] $0$-map graphs admit clique-based separators of size $\mathcal{O}(\sqrt{n})$; 
    \item[(b)] Unit disk graphs admit clique-based separators of size $\mathcal{O}(\sqrt{n})$;
    \item[(c)] Hyperbolic uniform disk graphs with radius $r$ admit clique-based separators of size $\mathcal{O}(\sqrt{\frac{r}{\tanh r}}\cdot \sqrt{n})$;
    \item[(d)] Spherical uniform disk graphs with radius $r$ admit clique-based separators of size $\mathcal{O}(\sqrt{n})$. 
\end{itemize}    
\end{corollary}

\begin{proof} All results are obtained by applying \Cref{cliquebased_layeredtreealpha} and \Cref{thetab_cliquecoverdeg}, and recalling the following bounds for the clique cover degeneracy and layered tree-independence number of the corresponding graph classes. $0$-map graphs have clique cover degeneracy at most $3$ (\Cref{0-map-theta-degen}) and layered tree-independence number at most $9$ (\Cref{maplayered}). Unit disk graphs have clique cover degeneracy at most $3$ \cite{Pee91} and layered tree-independence number at most $3$ \cite{GMY24}. Hyperbolic uniform disk graphs with radius $r$ have clique cover degeneracy at most $3$ \cite{BvdH24} and layered tree-independence number $\mathcal{O}(\frac{r}{\tanh r})$ (\Cref{layeredhyperbolic}). Spherical uniform disk graphs with radius $r$ have clique cover degeneracy at most $6$ (\Cref{sphericalcliquecover}) and layered tree-independence number at most $30$ (\Cref{sphericallayered}).  
\end{proof}

A remark about \Cref{cliquebasedcor} is in place. The bounds for $0$-map graphs and unit disk graphs are tight, as can be easily seen by considering the class of $\sqrt{n}\times \sqrt{n}$ grids. 
These bounds were first proven in \cite{dBKMT23} and \cite{BBK20}, respectively, although our proofs are arguably simpler. Recalling that $\text{UDG} \subseteq \text{HUDG}$ and $\text{UDG} \subseteq \text{SUDG}$, we also observe that the bounds for  hyperbolic and spherical uniform disk graphs are tight for some values of $r$. The significance of \Cref{cliquebasedcor} is in showing that layered tree-independence number arguments could be a useful tool in addressing clique-based separators.    

We finally show a general result encompassing the graph classes mentioned above.

\begin{proposition}\label{layeredtreealpha-cbs} The following (incomparable) graph classes of bounded layered tree-independence number admit clique-based separators of sublinear weight:    
\begin{itemize}
    \item[(a)] Every class of bounded layered tree-independence number with a subquadratic $\theta$-binding function;
    \item[(b)] Every class of bounded layered treewidth;
    \item[(c)] Every class of bounded tree-independence number.
\end{itemize}
\end{proposition}

\begin{proof} (a) immediately follows from \Cref{cliquebased_layeredtreealpha}. (b) follows from the fact that a class of bounded layered treewidth has balanced separators of size $\mathcal{O}(\sqrt{n})$ \cite{DMW17}. 

Finally, consider (c). Let $r \in \mathbb{N}$ be a constant and let $\mathcal{G}$ be a graph class such that $\tin(G) \leq r$ for every $G \in \mathcal{G}$. Let $G \in \mathcal{G}$ and let $\mathcal{T} = (T, \{X_t\}_{t\in V(T)})$ be a tree decomposition of $G$ with independence number at most $r$. We know that there exists a balanced separation $(A, B)$ in $G$ such that $A \cap B = X_t$, for some bag $X_t$ of $T$. It is then enough to show that every $n$-vertex graph $H$ with $\alpha(H) \leq r$ satisfies $\theta(H) \leq \frac{n}{n^{1/2r}- r}+\sqrt{n}$. The following argument is due to Buci\'c (personal communication, 2024).

It is well known that $R(p,q) \le \binom{p+q-2}{p-1}$ for all positive integers $p$ and $q$. Hence, $R(s, r+1) \leq \binom{s+r-1}{s-1} = \binom{s+r-1}{r} \leq \frac{(s+r-1)^r}{r!} \leq (s+r)^r$, and so every $n$-vertex graph with no independent set of size $r+1$ contains a clique of size at least $n^{1/r} - r$. Now, take a maximal collection $\mathcal{C}$ of pairwise-disjoint cliques in $H$ each of size $\lceil n^{1/2r}-r\rceil$. These cliques cover all vertices of $H$ except for at most $n^{1/2}$ vertices. Indeed, if more than $n^{1/2}$ vertices were left uncovered, Ramsey's theorem applied to the subgraph induced by such vertices gives a clique of size at least $n^{1/2r}-r$, contradicting the maximality of $\mathcal{C}$. We conclude that $\theta(H) \leq \frac{n}{n^{1/2r}- r}+\sqrt{n}$.
\end{proof}

\section{Bounded independence degeneracy}\label{sec:indep_deg}

In this section we consider \Cref{fractional tin-fragility bounded independence degeneracy}, which we restate for convenience. 

\conj*

We first provide some positive evidence by showing that the following fractionally $\tin$-fragile classes have bounded independence degeneracy: every class with bounded layered tree-independence number and every class of intersection graphs of fat objects in $\mathbb{R}^d$, for fixed $d$. 

\begin{lemma}\label{bounded layered tree-alpha bounded alpha-degeneracy}
    For every integer $k$, the class of graphs with layered tree-independence number at most~$k$ has independence degeneracy at most $3k$.
\end{lemma}

\begin{proof}
Let $G$ be a graph with layered tree-independence number at most $k$.
Fix a tree decomposition $\mathcal{T} = (T,\{X_t\}_{t \in V(T)})$ of $G$ and a layering $(V_0, V_1, \dots)$ of $V(G)$ such that, for each bag $X_t$ and each layer $V_i$, it holds $\alpha(X_t \cap V_i) \leq k$.
By \cite[Lemma~2.6]{DMS24}, there exist vertices $u \in V(G)$ and $t \in V(T)$ such that $N[u] \subseteq X_t$.
Let $i\ge 0$ be an integer such that $u \in V_i$.
Note that $N[u] \subseteq V_{i-1} \cup V_i \cup V_{i+1}$ (with $V_{-1} = \varnothing$ if $i= 0$).
Therefore, $\alpha(N[u]) \leq \alpha(N[u] \cap V_{i-1}) + \alpha(N[u] \cap V_{i}) + \alpha(N[u] \cap V_{i+1}) \leq \alpha(X_t \cap V_{i-1}) + \alpha(X_t \cap V_{i}) + \alpha(X_t \cap V_{i+1}) \leq 3k$. We infer that the class of graphs with layered tree-independence number at most $k$ has independence degeneracy at most $3k$.
\end{proof}

\begin{remark}\label{cliquecoverdegpower} An analogous argument shows that if a graph $G$ admits a tree decomposition and a layering such that each intersection between a bag and a layer has clique cover number at most $k$, then $G$ has clique cover degeneracy at most $3k$. Combining this with \Cref{layeredcliquepower} gives that powers of graphs with bounded layered treewidth have bounded clique cover degeneracy.
\end{remark}

Let $d \geq 2$ be an arbitrary but fixed integer. An \textit{object} in $\mathbb{R}^d$ is a path-connected compact set $O \subset \mathbb{R}^d$. The \textit{size} of an object $O$ in $\mathbb{R}^d$, denoted $s(O)$, is the side length of its smallest enclosing axis-aligned hypercube. Let $c\in \mathbb{R}$ be a positive constant. A collection $\mathcal{O}$ of objects in $\mathbb{R}^d$ is $c$-$\textit{fat}$ if, for every $r\in \mathbb{R}$ and every axis-aligned closed hypercube $R$ of side length $r$, there exist at most $c$ pairwise non-intersecting objects from $\mathcal{O}$ of size at least $r$ and which intersect $R$. It is known that every class of intersection graphs of a $c$-fat collection of objects in $\mathbb{R}^d$, for fixed $d$, is fractionally $\tin$-fragile \cite{GMY23}.

\begin{lemma}\label{c-fat bounded alpha-degeneracy}
    For every $c>0$, the class of intersection graphs of a $c$-fat collection of objects in $\mathbb{R}^d$ has independence degeneracy at most $c$.
\end{lemma}
\begin{proof}
Let $\mathcal{O}$ be a $c$-fat collection of objects in $\mathbb{R}^d$ and let $G$ be its intersection graph.
Take a vertex $u \in V(G)$ corresponding to an object $O_u \in \mathcal{O}$ with minimum size and denote by $r$ the size of $O_u$. Fix a closed axis-aligned hypercube $R$ of side length $r$ that entirely contains $O_u$.
Observe that, by the choice of $u$, every neighbor $v$ of $u$ corresponds to an object $O_v$ of size at least $r$ that intersects $R$.
Therefore, since $\mathcal{O}$ is $c$-fat, the number of pairwise non-intersecting objects that intersect $R$ is at most~$c$.
We deduce that $u$ has at most $c$ pairwise non-adjacent neighbors, and this implies that $G$ has independence degeneracy at most~$c$.
\end{proof}

We now prove a relaxation of \Cref{fractional tin-fragility bounded independence degeneracy}, namely that every fractionally $\tin$-fragile class is  polynomially $(\dg,\omega)$-bounded. Let us first observe the following.

\begin{lemma}\label{alpha-degeneracy dg-omega-boundedness}
    For every integer $k$, the class of graphs with independence degeneracy at most $k$ is polynomially $(\dg,\omega)$-bounded.
\end{lemma}

\begin{proof}
The lemma follows from Ramsey's theorem.
Fix an integer $k\ge 0$ and let $f_k$ be a function defined on the set of nonnegative integers by the rule  $f_k(t) = R(t,k+1)-1$. Recall that $R(p,q) \le \binom{p+q-2}{p-1}$ for all positive integers $p$ and $q$.
Hence, $f_k(t)\le \binom{t+k-1}{k}$ is bounded from above by a polynomial in $t$ of degree $k$.
Consequently, since the class of graphs with independence degeneracy at most $k$ is hereditary, it suffices to show that every graph $G$ with independence degeneracy at most $k$ satisfies $\dg(G)\le f_k(\omega(G))$.
Let $G$ be a graph with independence degeneracy at most $k$ and let $t = \omega(G)$. Note that every nonnull induced subgraph $H$ of $G$ has a vertex $v\in V(H)$ such that $\alpha(N_H(v))\le k$.
Since the clique number of the graph $H[N_H(v)]$ is at most $\omega(G)-1 = t-1$, Ramsey's theorem implies that $d_H(v)\le R(t,k+1)-1 = f_k(t)$. We have thus shown that every nonnull induced subgraph of $G$ has a vertex of degree at most $f_k(t)$, that is, the degeneracy of $G$ is at most $f_k(t)$.
\end{proof}

\begin{remark} \Cref{alpha-degeneracy dg-omega-boundedness} could be a useful tool in bounding the degeneracy of geometric intersection graphs of bounded ply. For example, Lokshtanov et al.~\cite{LPSXZ24} showed that every $K_r$-free pseudo-disk graph has $\mathcal{O}(r)$ degeneracy. Since pseudo-disk graphs have independence degeneracy at most $156$, as shown by Pinchasi~\cite{Pin14}, \Cref{alpha-degeneracy dg-omega-boundedness} implies that they have degeneracy bounded by a polynomial in $r$ of degree $156$. 
\end{remark}

Dvo\v{r}{\'{a}}k~\cite{Dvo16} showed that fractional $\tw$-fragility implies bounded expansion and hence bounded degeneracy. We prove in \Cref{bounded tw-fragility bounded degeneracy-new} below a slightly better bound for the degeneracy than the one that can be obtained from \cite[Lemma~13]{Dvo16}. For the proof, recall that treewidth, average degree (denoted by $\ad(G)$), and degeneracy are related as follows (see, e.g.,~\cite{BWK06} and~\cite[Corollary 1]{CS05}, respectively).

\begin{lemma}\label{average degree and treewidth}
Every nonnull graph $G$ satisfies $\ad(G)\le 2\dg(G)$ and $\dg(G)\le \tw(G)$.
\end{lemma}

\begin{lemma}\label{bounded tw-fragility bounded degeneracy-new}
Let $r>2$, $k\ge 0$, and let $G$ be a graph having a $(1-1/r)$-general cover $\mathcal{C}$ such that $\tw(G[A]) \leq k$ for every $A \in \mathcal{C}$.
Then, the degeneracy of $G$ is at most $\frac{2rk}{r-2}$.
\end{lemma}

\begin{proof}
Let $r$, $k$, $G$, and $\mathcal{C}$ be as in the statement of the theorem and suppose for a contradiction that the degeneracy of $G$ is more than $\frac{2rk}{r-2}$.
By the definition of degeneracy, $G$ contains an induced subgraph $H$ with minimum degree more than $\frac{2rk}{r-2}$. 
Note that $\mathcal{C}_H \coloneqq \{C\cap V(H)\colon C\in \mathcal{C}\}$ is a  $(1-1/r)$-general  cover of $H$ and for any $A \in \mathcal{C}_H$ with $A  = C\cap V(H)$ for some $C\in \mathcal{C}$, $\tw(H[A]) \le \tw(G[C]) \leq k$.

Pick $A \in \mathcal{C}_H$ uniformly at random. 
Since $\mathcal{C}_H$ is a $(1-1/r)$-general cover of $H$, we have $\mathbb{P}(v \not\in A)\le 1/r$ for every vertex $v\in V(H)$.
Furthermore, for each $v\in V(H)$ and each $u \in N_H(v)$, it holds that
\[
\mathbb{P}\left(\{u \not\in A\} \text{ or } \{v\not \in A\}\right) \leq  \mathbb{P}(u \not\in A) + \mathbb{P}(v \not\in A) \\
\leq \frac{2}{r}.
\]

Hence, for each $v\in V(H)$, we have
\begin{align*}
\mathbb{E}\left[\mathbbm{1}_{\{v\in A\}}\cdot|N_{H[A]}(v)|\right] 
&\geq  \mathbb{E}\left[\mathbbm{1}_{\{v\in A\}}\cdot|N(v) \cap A|\right]  \\
&\geq \mathbb{E}\left[\sum_{u\in N_H(v)}\mathbbm{1}_{\{u,v \in A\}}\right] \\
&\geq \sum_{u \in N_H(v)}(1-\mathbb{P}\left(\{u \not\in A\} \text{ or } \{v\not \in A\}\right)) \\
&\geq \left(1-\frac{2}{r}\right)\cdot|N_H(v)|.
\end{align*}
Then, 
\[
\mathbb{E}\left[\sum_{v\in V(H)}\left[\mathbbm{1}_{\{v\in A\}}\cdot|N_{H[A]}(v)|\right]\right]
= \sum_{v\in V(H)}\mathbb{E}\left[\mathbbm{1}_{\{v\in A\}}\cdot|N_{H[A]}(v)|\right] 
\geq \left(1-\frac{2}{r}\right)\cdot \sum_{v\in V(H)}|N_H(v)|.
\]
Therefore, there exists some $A \in \mathcal{C}_H$ such that \[\sum_{v\in V(H)}\left[\mathbbm{1}_{\{v\in A\}}\cdot |N_{H[A]}(v)|\right] \geq \left(1-\frac{2}{r}\right)\cdot \sum_{v\in V(H)}|N_H(v)|\,.\] 
This implies, in particular, that $A\neq \varnothing$ and
\begin{align*}
\frac{1}{|A|}\sum_{v\in V(H[A])}|N_{H[A]}(v)|&=
\frac{1}{|A|}\sum_{v\in V(H)}\left[\mathbbm{1}_{\{v\in A\}}\cdot |N_{H[A]}(v)|\right]\\
&\geq \frac{1}{|V(H)|}\sum_{v\in V(H)}\left[\mathbbm{1}_{\{v\in A\}}\cdot |N_{H[A]}(v)|\right]\\
&\geq
\left(1-\frac{2}{r}\right)\cdot \frac{1}{|V(H)|} \sum_{v\in V(H)}|N_H(v)|,
\end{align*}
that is, $\ad(H[A]) \geq \left(1-\frac{2}{r}\right)\ad(H)$. 
Then, as the minimum degree of $H$ is more than $\frac{2rk}{r-2}$, we have that $\ad(H) > \frac{2rk}{r-2}$. 
Hence, $\ad(H[A]) > \left(1-\frac{2}{r}\right)\cdot\frac{2rk}{r-2}=2k$.
By \Cref{average degree and treewidth}, we conclude that $\tw(H[A]) > k$, a contradiction.
\end{proof}

\begin{corollary}\label{fractional tin-fragility poly chi-bounded}
Every fractionally $\tin$-fragile graph class is  polynomially $(\dg,\omega)$-bounded. 
\end{corollary}

\begin{proof}
Let $\mathcal{G}$ be a fractionally $\tin$-fragile graph class and let $f$ be a function such that for every $r\in \mathbb{N}$, every graph $G\in \mathcal{G}$ has a $(1-1/r)$-general cover $\mathcal{C}$ such that $\tin(G[A]) \leq f(r)$ for every $A \in \mathcal{C}$.
For a positive integer $k$, let $\mathcal{G}_k$ denote the class of all graphs in $\mathcal{G}$ with clique number at most $k$.
Let $G$ be a graph in $\mathcal{G}_k$ and let $\mathcal{C}$ be a $(1-1/3)$-general cover of $G$ such that $\tin(G[A]) \leq f(3)$ for every $A \in \mathcal{C}$.
Consider an arbitrary $A \in \mathcal{C}$.
Since $\tin(G[A]) \leq f(3)$ and $\omega(G[A])\le \omega(G)\le k$, \Cref{ramsey} implies that the treewidth of the graph $G[A]$ is at most $R(k+1,f(3)+1)-2$. 
By \Cref{bounded tw-fragility bounded degeneracy-new}, the degeneracy of $G$ is at most $6R(k+1,f(3)+1)-12$, which is a function of $k$ bounded from above by a polynomial of degree $f(3)$.
\end{proof}

\section{Concluding remarks and open problems}

In this paper we provided evidence toward a positive answer to \Cref{questionltreealpha} and \Cref{fractional tin-fragility bounded independence degeneracy} which, however, remain open in general. We now list some further questions arising from our work. 

We showed in \Cref{layeredhyperbolic} that every hyperbolic uniform disk graph with radius $r$ has layered tree-independence number $\mathcal{O}(\frac{r}{\tanh r})$. It is natural to ask whether this can be strengthened to $\mathcal{O}(1)$ layered tree-independence number. 

\begin{question}\label{hudgconstantlayered}
    Do hyperbolic uniform disk graphs have bounded layered tree-independence number? 
\end{question}

A positive answer to \Cref{hudgconstantlayered} would imply that hyperbolic uniform disk graphs have bounded local tree-independence number and hence, thanks to \Cref{ramsey}, that the treewidth of the neighborhood of every vertex of a hyperbolic uniform disk graph is bounded by a function of the clique number. 
Whether this is the case has been asked by Bläsius and Merker~\cite{BM}.

We observed in \Cref{bounded layered tree-alpha bounded alpha-degeneracy} that bounded layered tree-independence number implies bounded independence degeneracy and recall that \Cref{questionltreealpha} asks whether every graph class of bounded layered tree-independence number admits clique-based separators of sublinear weight. We then ask the following (see also \Cref{reldiagram}).

\begin{question}\label{clique-based_vs_ideg}
Is it true that every graph class admitting clique-based separators of sublinear weight has bounded independence degeneracy?
\end{question}

Extending a result of Ye and Borodin~\cite{YB12}, we showed in \Cref{0-map-theta-degen} that $0$-map graphs have clique cover degeneracy at most $3$. This bound is tight, as there exist planar graphs with clique cover degeneracy $3$. It is natural to ask what happens for $g \geq 1$. 

\begin{question}\label{openmap}
    Do $g$-map graphs have bounded clique cover degeneracy? 
\end{question}

In order to address \Cref{openmap}, one could perhaps use the fact that graphs with bounded genus have bounded degeneracy. A positive answer to \Cref{openmap} would imply that $g$-map graphs admit clique-based separators of size $\mathcal{O}(\sqrt{gn})$.

Recall that $\text{HUDG} \subsetneq \text{DG}$ (see, e.g., \cite{BvdH24}). We ask whether a similar relation holds for spherical uniform disk graphs. There exist disk graphs which are not spherical uniform disk graphs, as implied for example by \Cref{sphericallayered}. Moreover, if the family of spherical disks in a spherical uniform disk graph $G$ does not cover the unit sphere, then applying a stereographic projection gives a family of Euclidean disks whose intersection graph is isomorphic to $G$. However, it is not clear how to deal with families of disks covering the sphere.

\begin{question}\label{quest:SUDG-DG}
    Is it true that $\emph{SUDG} \subseteq \emph{DG}$?
\end{question}

\paragraph{Acknowledgements.} We thank Matija  Bucić for suggesting an argument used in the proof of \Cref{layeredtreealpha-cbs}.
This research was supported by the University of Parma through the action ``Bando di Ateneo 2023 per la ricerca''. 
A.~M. is a member of the Gruppo Nazionale Calcolo Scientifico-Istituto Nazionale di Alta Matematica (GNCS-INdAM).
The work of M.~M.~was supported in part by the Slovenian Research and Innovation Agency (I0-0035, research program P1-0285 and research projects J1-3003, J1-4008, J1-4084, J1-60012, and N1-0370) and by the research program CogniCom (0013103) at the University of Primorska.

\bibliographystyle{plainnat}
\bibliography{biblio}

\end{document}